\documentclass[10pt]{amsart}

\usepackage{amssymb}
\usepackage{booktabs}
\usepackage[shortlabels]{enumitem}
\usepackage{eucal}
\usepackage[margin=1in]{geometry}
\usepackage[colorlinks=true, allcolors=blue]{hyperref}

\newtheorem{theorem}{Theorem}[section]

\newtheorem{corollary}[theorem]{Corollary}
\newtheorem*{corollary*}{Corollary}

\newtheorem{lemma}[theorem]{Lemma}
\newtheorem*{lemma*}{Lemma}
\newtheorem{proposition}[theorem]{Proposition}
\newtheorem*{proposition*}{Proposition}
\newtheorem*{theorem*}{Theorem}

\theoremstyle{definition}
\newtheorem{definition}[theorem]{Definition}

\numberwithin{equation}{section}

\newcommand{\bb}[1]{\text{$\mathbb{#1}$}}
\newcommand{\C}{\mathbb{C}}
\newcommand{\cal}[1]{\text{$\mathcal{#1}$}}
\DeclareMathOperator{\ch}{ch}
\renewcommand{\epsilon}{\varepsilon}
\DeclareMathOperator{\ev}{ev}

\newcommand{\hlie}[1]{\text{$\hat{\mathfrak{#1}}$}}
\DeclareMathOperator{\lng}{lg}
\newcommand{\lie}[1]{\text{$\mathfrak{#1}$}}
\DeclareMathOperator{\sh}{sh}
\DeclareMathOperator{\vspan}{span}
\newcommand{\tlie}[1]{\text{$\tilde{\mathfrak{#1}}$}}
\DeclareMathOperator{\wt}{wt}

\begin{document}

\title{Demazure and local Weyl modules for \\ twisted hyper current algebras}
\author{A. Bianchi}
\author{T. Macedo}
\address{Department of Science and Technology\\
         Federal University of S\~ao Paulo\\
         S\~ao Jos\'e dos Campos, S\~ao Paulo, Brazil, 12.247-014}
\email{acbianchi@unifesp.br, tmacedo@unifesp.br}
\thanks{Partially supported by the FAPESP grant 2015/22040-0 (A.B) and the CNPq grant 462315/2014-2 (A.B. and T.M.)}

\begin{abstract}
In this paper, we study local graded Weyl modules and Demazure modules for twisted hyper current algebras.  We prove that local graded Weyl modules for a twisted hyper current algebra are isomorphic to the corresponding level 1 Demazure modules, and moreover, that they are restrictions of corresponding local graded Weyl modules for the untwisted hyper current algebra.
\end{abstract}

\maketitle

%%%%%%%%%%%%%%%%%%%%%%%
\section*{Introduction}

The main goal of this paper is to generalize some results relating local Weyl and Demazure modules for twisted current algebras, which are known to hold over an algebraically closed field of characteristic zero, to the positive characteristic setting.  Namely, in \cite[Theorem 2.i]{CFS08}, it was proved that any local Weyl module for a twisted affine algebra is isomorphic to a certain restriction of a local Weyl module for the corresponding untwisted affine algebra.  In \cite[Theorem 6.0.2]{FK13}, it was proved that any graded local Weyl module for a twisted current algebra is isomorphic to a certain graded restriction of a graded local Weyl module for the corresponding untwisted current algebra.  It is interesting to notice that these results were also obtained by \cite[Lemma 3.8]{FKKS12} in a more general framework of equivariant map algebras.

In the simply laced case, it was proved in \cite[Proposition 2]{foli:weyldem} that any local Weyl module for a current algebra is isomorphic to a certain automorphic image of a local Weyl module for the corresponding loop algebra.  In \cite{CP01}, for $\lie{sl}_2$, and in \cite{CL06}, for type $A$, and \cite[Theorem 7]{foli:weyldem}, for types ADE, both using the $\lie{sl}_2$-case, it was proved that any local graded Weyl module for a current algebra is isomorphic to a certain restriction of a Demazure module for the corresponding affine algebra.  This latter result was shown to be false in the non simply laced case.  In fact, it was proved that local graded Weyl modules for current algebras admit Demazure flags in \cite[Proposition 4.18]{naoi:weyldem}.  Finally, it was proved in \cite[Theorem 5.0.2]{FK13} that, under some somewhat restrictive conditions, any twisted graded local Weyl module for a current algebra is isomorphic to a certain twisted Demazure module.

Some positive characteristic generalizations of the results mentioned above are already known.  For instance, it was proved in \cite[Theorem 4.1]{BM14} that any local Weyl module for a twisted hyper loop algebra is isomorphic to a certain restriction of a local Weyl module for the corresponding untwisted hyper loop algebra. In the simply laced case, it was proved in \cite[Theorem 1.5.2]{BMM15} that any local Weyl module for a hyper current algebra is isomorphic to a certain automorphic image of a local Weyl module for the corresponding hyper loop algebras, and that any local graded Weyl module for a current algebra is isomorphic to a certain restriction of a Demazure module for the corresponding affine algebra.

It is important to remark that, in the non simply laced case, it was proved in \cite[Theorem 1.5.2.b]{BMM15} that any local Weyl module for a hyper current algebra admits a Demazure flag.  Since Dynkin diagrams associated to non simply laced simple Lie algebras admit only the trivial automorphism, the following   claim is an immediate consequence of \cite[Theorem 1.5.2]{BMM15}:  in the non simply laced case, every local graded Weyl module for a twisted hyper current algebra admits a Demazure flag, and moreover, it is isomorphic to a certain local Weyl module for the twisted hyperloop algebra.

Hence, our main goal in this paper is to prove the following result.
\begin{theorem*}
The local graded Weyl module of highest weight $\lambda$ for a twisted hyper current algebra is isomorphic to the corresponding level 1 Demazure module.  And, moreover, it is a restriction of the corresponding local graded Weyl module of highest weight $\lambda$ for the untwisted hyper current algebra.
\end{theorem*}
The proof of the first part is given in Theorem~\ref{t:wd} and the proof of the second part is given in Theorem~\ref{thm.2}.  Along the way, we also prove some results about local graded Weyl modules and Demazure modules for twisted hyper current algebras.  For instance, in Lemma~\ref{lem:WZint}, we prove that integral local graded Weyl modules are integrable modules, and in Proposition~\ref{prop:WZfg}, we prove that they are finitely generated.  Also, in Proposition~\ref{p:Demrels}, we prove that the purely algebraic definition of a Demazure module given here reflects the usual geometric definition of Demazure modules.

In terms of equivariant map algebras, $\lie g [t]^\sigma$ can be regarded as the Lie algebra of equivariant regular maps $\varphi \colon \bb C \to \lie g$ under the cyclic group generated by the automorphism $\sigma$. Here, the cyclic group generated by $\sigma$ acts on \lie g as in Section \ref{s:simple} and on $\bb C$ via multiplication by roots of unity. Observe that this latter action is not free, since $0 \in \bb C$ is a fixed point. Thus, in the current paper, we are dealing with a case that is not covered by \cite{FKKS12}.  Moreover, we generalize some of the results in \cite{FKKS12} to fields of positive characteristic (different from 2 in case $\lie g \cong A_{2n}$).

\subsection*{Notation} Denote by $\mathbb C, \mathbb Z, \mathbb Z_{\ge0}$ and $\mathbb  N$ the set of complex numbers, the set of integers, non–negative integers and positive integers, respectively. Fix an indeterminate $t$ and let $\mathbb C[t]$ (resp. $\mathbb C[t, t^{-1}]$) be the corresponding polynomial ring (resp. Laurent polynomial ring) with complex coefficients.  Throughout this paper, unless we explicitly state otherwise, all vector spaces and tensor products are assumed to be taken over $\mathbb C$.

%%%%%%%%%%%%%%%%%%%%%%%%%%%%%%%%%%%%%%%%%%%%%%%%%
\section{Preliminaries on algebras}\label{s:algs}

%%%%%%%%%%%%%%%%%%%%%%%%%%%%%%%%%%%%%%%%%%%%%%%%%%%%%%%%%%%%%%%%%%%%%%%%%%
\subsection{Simple Lie algebras and diagram automorphisms}\label{s:simple}

Let $I$ be the set of vertices of a finite-type connected Dynkin diagram and let $\lie g$ be the associated finite-dimensional simple Lie algebra over $\mathbb C$. Fix a Cartan subalgebra $\lie h \subseteq \lie g$, a triangular decomposition $\lie g=\lie n^-\oplus\lie h\oplus\lie n^+$, and denote by $R$ (resp. $R^+$) the associated root system (resp. set of positive roots). Let $\{\alpha_i \mid i\in I \}$ (resp. $\{ \omega_i \mid i\in I \}$) denote the set of simple roots (resp. fundamental weights), and let $Q=\bigoplus_{i\in I} \mathbb Z \alpha_i$, $Q^+=\bigoplus_{i\in I} \mathbb Z_{\ge0} \alpha_i$, $P=\bigoplus_{i\in I} \mathbb Z \omega_i$, $P^+=\bigoplus_{i\in I} \mathbb Z_{\ge0} \omega_i$.  Denote by $\cal W$ the Weyl group of $\lie g$.  Let $\lie g_\alpha$ denote the root space of $\lie g$ associated to $\alpha \in R$ and $\theta$ denote the highest root of $\lie g$.

Fix a Chevalley basis $\cal C=\{x_\alpha^\pm, h_i \mid \alpha\in R^+, \, i\in I\}$ of $\lie g$ and, for each $\alpha \in R^+$, let $h_\alpha = [x^+_\alpha, x^-_\alpha]$ (in particular, $h_{\alpha_i} = h_i$ for all $i \in I$).  There exists a unique bilinear form $(\ ,\ )$ on $\lie g$ that is symmetric, invariant, nondegenerate, and such that $(h_\theta,h_\theta)=2$.  Let $\nu \colon \lie h \to \lie h^*$ be the linear isomorphism induced by $(\ ,\ )$ and keep denoting by $(\ ,\ )$ the nondegenerate bilinear form induced by $\nu$ on $\lie h^*$.  Denote by $r^\vee \in \{1,2,3\}$ the lacing number of $\lie g$, and let
\begin{equation}
r^\vee_\alpha = \frac{2}{(\alpha,\alpha)} = \begin{cases}
1, &\text{if $\alpha$ is long},\\
r^\vee, &\text{if $\alpha$ is short.}
\end{cases}
\end{equation}

Fix an automorphism $\sigma$ of the Dynkin diagram of $\lie g$ and let $m\in\{1,2,3\}$ denote the order of $\sigma$.  There exists a unique Lie algebra automorphism of $\lie g$ (which we keep denoting by $\sigma$) that satisfies $\sigma(x^\pm_i)=x^\pm_{\sigma(i)}$ for all $i\in I$.  Define a permutation of $R$ (which we also denote by $\sigma$) by
\[
\sigma \left( \sum_{i \in I} n_i \alpha_i \right)
:= \sum_{i \in I} n_i \alpha_{\sigma(i)}.
\]
For each $\alpha\in R^+$, define $m_\alpha := \# \left\{\sigma^j(\alpha) \mid j\in \bb N \right\}$.  Observe that $m_\alpha\in\{1,m\}$ and $\sigma(\lie g_\alpha)=\lie g_{\sigma(\alpha)}$ for all $\alpha\in R$.

Fix a primitive $m$-th root of unity $\zeta$, and let $\lie g_\epsilon = \{x\in\lie g \mid \sigma(x)=\zeta^\epsilon x\}$ for each $\epsilon \in \{0, \dotsc, m-1\}$.  Since every finite order automorphism is diagonalizable over $\mathbb C$ (with eigenvalues being roots of unity), we have
\begin{equation*}
\lie g=\bigoplus_{\epsilon=0}^{m-1}\lie g_\epsilon \quad\text{and}\quad [\lie g_\epsilon, \lie g_{\epsilon'}]\subseteq \lie g_{\overline{\epsilon+\epsilon'}},
\end{equation*}
where $\overline{\epsilon+\epsilon'}$ is the remainder of the division of $\epsilon+\epsilon'$ by $m$.  (Henceforth, we will abuse notation and write $\epsilon+\epsilon'$ instead of $\overline{\epsilon+\epsilon'}$.)  In particular, $\lie g_0$ is a subalgebra of $\lie g$ and $\lie g_\epsilon$ is a $\lie g_0$-module for all $\epsilon \in \{1, \dotsc, m-1\}$.  Moreover, it is known that $\lie g_0$ is a finite-dimensional simple Lie algebra and $\lie g_\epsilon$ is a finite-dimensional irreducible $\lie g_0$-module for all $\epsilon \in \{0, \dotsc, m-1\}$ (see \cite[Chapter 8]{Kac90}).

If $\lie a$ is a subalgebra of $\lie g$ and $\epsilon \in \{0, \dotsc, m-1\}$, let $\lie a_\epsilon:=\lie a\cap\lie g_\epsilon$. It is known that $\lie h_0$ is a Cartan subalgebra of $\lie g_0$ and $\lie g_0= \lie n^+_0 \oplus \lie h_0 \oplus \lie n^-_0$ is a triangular decomposition. Let $R_0$ denote the root system determined by $\lie h_0$, and $I_0$ be an indexing set for the associated simple roots.  The simple roots (resp. fundamental weights) of $\lie g_0$ determined by the triangular decomposition $\lie g_0= \lie n^+_0 \oplus \lie h_0 \oplus \lie n^-_0$ will be denoted by $\alpha_i$, $i\in I_0$ (resp. $\omega_i$, $i\in I_0$), as this should cause no confusion.  Let $\theta_0$ denote the highest root of $\lie g_0$, $Q_0 = \bigoplus_{i\in I_0} \mathbb Z \alpha_i$, $Q_0^+=\bigoplus_{i\in I_0} \mathbb Z_{\ge0} \alpha_i$, $P_0=\bigoplus_{i\in I_0} \mathbb Z \omega_i$, $P_0^+=\bigoplus_{i\in I_0} \mathbb Z_{\ge0} \omega_i$.  Let $R_{\sh}$ (resp. $R_{\lng}$) be the subset of $R_0^+$ corresponding to short (resp. long) positive roots of $\lie g_0$, and $\cal W_0$ be the Weyl group of $\lie g_0$.

For each $\lie g$-module (resp. of $\lie g_0$-module) $V$, denote by $V_\lambda$, $\lambda \in \lie h^*$ (resp. $\lambda \in \lie h_0^*$), its weight space of weight $\lambda$.  Let $\wt(V)$ (resp. $\wt_0(V)$) denote the set of weights of $V$ as an $\lie h$-module (resp. $\lie h_0$-module).  For each $\epsilon \in \{0, \dotsc, m-1 \}$, denote by $R_\epsilon$ the set of weights of $\lie g_\epsilon$ as a $\lie g_0$-module (via the adjoint representation).  It is known that $R_\epsilon \subseteq Q_0$, so let $R_\epsilon^+ = (R_\epsilon \cap Q_0^+) \setminus \{0\}$.  Finally, denote by $\theta_1$ the highest weight of $\lie g_1$ (as a $\lie g_0$-module).

\iffalse
\begin{table}
\begin{tabular}{lccc}
\toprule
$\lie g$ & $m$ & $\lie g_0$ & $\theta_1$ \\
\midrule
$A_{2n}$ & 2 & $B_n$ & $2 \theta_0$ \\
$A_{2n-1}$ & 2 & $C_n$ & $\theta_0$ \\
$D_{n+1}$ & 2 & $B_n$ & $\theta_0$ \\
$D_4$ & 3 & $G_2$ & $\theta_0$ \\
$E_6$ & 2 & $F_4$ & $\theta_0$ \\
\bottomrule
\end{tabular}
\end{table}
\fi

%%%%%%%%%%%%%%%%%%%%%%%%%%%%%%%%%%%%%%%%%%%%%%%%%%%%%%%%%%%%%%
\subsection{Chevalley basis of $\lie g_0$} \label{s:chevalley}

Given $\epsilon \in \{0, \dotsc, m-1\}$ and $\mu \in R_\epsilon \setminus \{0\}$, it is known that: $\mu=\alpha|_{\lie h_0}$ for some $\alpha\in R$, and $\alpha|_{\lie h_0}=\beta|_{\lie h_0}$ if, and only if, $\beta=\sigma^j(\alpha)$ for some $j \in \{0, \dotsc, m_\alpha-1\}$.

Suppose $\lie g$ is not of type $A_{2n}$. Given $\alpha\in R^+$ and $\epsilon \in \{0, \dotsc, m_\alpha-1\}$, let
\[
x_{\alpha,\epsilon}^\pm =
\sum_{j=0}^{m_\alpha-1} \zeta^{j\epsilon}x^\pm_{\sigma^j(\alpha)}
\qquad\textup{and}\qquad
h_{\alpha,\epsilon} =
\sum_{j= 0}^{m_\alpha-1} \zeta^{j\epsilon}h_{\sigma^j(\alpha)}.
\]
When $\epsilon \ge m_\alpha$, let $x^\pm_{\alpha, \epsilon} = h_{\alpha, \epsilon} = 0$.  Now, suppose $\lie g$ is of type $A_{2n}$.  Given $\alpha\in R^+$ and $\epsilon \in \{0, 1\}$, let
\[
x_{\alpha,\epsilon}^\pm :=
\begin{cases}
\delta_{1,\epsilon}x_\alpha^\pm,
& \quad\text{if } \alpha=\sigma(\alpha),\\
x_\alpha^\pm+(-1)^\epsilon x_{\sigma(\alpha)}^\pm,
& \quad\text{if } \alpha\ne\sigma(\alpha), \alpha|_{\lie h_0}\in R_{\lng},\\
\sqrt 2\left( x_\alpha^\pm+ (-1)^\epsilon x_{\sigma(\alpha)}^\pm\right),
& \quad\text{if } \alpha\ne\sigma(\alpha), \alpha|_{\lie h_0}\in R_{\sh},
\end{cases}
\]
and
\[
h_{\alpha,\epsilon} :=
\begin{cases}
\delta_{0,\epsilon}h_\alpha,
& \quad\text{if } \alpha=\sigma(\alpha),\\
h_\alpha+(-1)^\epsilon h_{\sigma(\alpha)},
& \quad\text{if }  \alpha\ne\sigma(\alpha), \alpha|_{\lie h_0}\in R_{\lng},\\
2 \left(h_\alpha+(-1)^\epsilon  h_{\sigma(\alpha)}\right),
& \quad\text{if }  \alpha\ne\sigma(\alpha), \alpha|_{\lie h_0}\in R_{\sh}.
\end{cases}
\]
Moreover, if $\lie g$ is of type $A_{2n}$ and $\beta+\sigma(\beta)\in R^+$, then $\beta|_{\lie h_0} \in R_{\sh}$ and $(\beta + \sigma(\beta))|_{\lie h_0} \in 2R_{\sh}$.  In this case, we have
\begin{equation}\label{e:x2Rs}
x_{\beta+\sigma(\beta),1}^\pm
= \frac{s}{4}[x_{\beta,0}^\pm,x_{\beta,1}^\pm]
= -s[ x_\beta^\pm,x_{\sigma(\beta)}^\pm]
\quad \textup{ for some $s \in \{ -1, 1\}$}.
\end{equation}
Since $\lie g_\epsilon$ is a finite-dimensional irreducible $\lie g_0$-module, for all $\epsilon \in \{0, \dotsc, m-1\}$ and $\mu \in R_\epsilon \setminus \{0\}$, we have
\[
\lie h_\epsilon = \vspan_\bb C \{ h_{\alpha_i,\epsilon} \mid i\in I\}
\qquad \textup{and} \qquad
(\lie g_\epsilon)_{\pm\mu} = \mathbb Cx_{\alpha,\epsilon}^\pm,
\quad \textup{where} \ \alpha|_{\lie h_0} = \mu.
\]

Choose a complete set of representatives $O$ of the orbits of the action of $\sigma$ on $R^+$.  For each $\mu \in R^+_\epsilon$, let $x_{\mu,\epsilon}^\pm = x_{\alpha,\epsilon}^\pm$, where $\alpha \in O$ is such that $\alpha|_{\lie h_0}=\mu$.  Also, notice that there exists a unique injective map $o \colon I_0\to I$ such that $\alpha_{o(i)}\in O$ for all $i\in I_0$, and let $h_{i,\epsilon} = h_{o(i),\epsilon}$.  Thus
\begin{equation}\label{eq:chev.basis.g}
\cal C^\sigma(O)
:= \left\{ x_{\mu,\epsilon}^\pm, \, h_{i,\epsilon} \mid \epsilon \in \{0, \dotsc, m-1\},\,  \mu \in R^+_\epsilon,\, i\in I_0 \right\} \setminus\{ 0 \}
\end{equation}
is a basis of $\lie g$.

For notational convenience, if $\alpha\in O$ is such that $\mu=\alpha|_{\lie h_0}$, let $h_{\mu,\epsilon}=h_{\alpha,\epsilon}$, and if $\mu \notin R_\epsilon$, let $x_{\mu,\epsilon}^\pm=0$.  Also, if $\lie g$ is of type $A_{2n}$, assume that $O$ is chosen in such a way that $s=1$ in \eqref{e:x2Rs}.  Moreover, if $\lie g$ is of type $D_4$ and $m=3$, let $j\in I$ be the unique vertex fixed by $\sigma$, choose $i\in I$ such that $\sigma(i)\ne i$, and let $O = \{\alpha\in R^+ \mid \sigma(\alpha)=\alpha\}\cup\{\alpha_i, \, \alpha_j+\alpha_i,\, \alpha_j+\sigma(\alpha_i)+\sigma^2(\alpha_i)\}$.  It has been proven by Kac that the set $\cal C^\sigma_0(O) := \{x_{\mu,0}^\pm, h_{i,0} \mid \mu\in R_0^+, i\in I_0\}$ is a Chevalley basis of $\lie g_0$ (see \cite[\S8.3]{Kac90}).

Later in this paper, we shall need the following formulas of commutators of elements of $\cal C^\sigma(O)$, proved in \cite[Lemma 2.2.7]{bianchi12}.  For all $\epsilon,\epsilon' \in \{0, \dotsc, m-1\}$, $\mu \in R_0^+$, $\nu \in R^+_\epsilon$, $\eta \in R^+_\epsilon \cap R^+_{\epsilon'}$:
\begin{align*}
[h_{\mu,0}, x^\pm_{\nu,\epsilon}] 
&= \pm \nu(h_{\mu,0}) x^\pm_{\nu,\epsilon}; \\
\textup{if $h_{\nu,1}\ne0$, then} \quad [h_{\nu,1},x_{\nu,\epsilon}^\pm] 
&= \begin{cases}
\pm 3 x^\pm_{\nu,\epsilon+1},
&\text{ if } \nu \in R_{\sh} \text{ and } \lie g \text{ is of type } A_{2n}, \\
\pm 2 x^\pm_{\nu,\epsilon+1},
& \text{ otherwise;} \end{cases}\\
[x_{\eta,\epsilon}^+, x_{\eta,\epsilon'}^-]
&= \begin{cases}
\delta_{\epsilon\epsilon',1}h_{\frac{\eta}{2},0},
& \textup{if $\eta \in 2R_{\sh}$ and $\lie g$ is of type $A_{2n}$}, \\
h_{\eta, \epsilon+\epsilon'},
& \textup{otherwise}.
\end{cases}
\end{align*}

%%%%%%%%%%%%%%%%%%%%%%%%%%%%%%%%%%%%%%%%%%%%%%%%%%%%%%%%
\subsection{Current, loop and affine Kac-Moody algebras}

Given a vector space $\lie a$, consider $\tlie a = \lie a \otimes \C[t, t^{-1}]$ and $\lie a[t] = \lie a \otimes \C [t]$.  If $(\lie a, [ \cdot, \cdot ]_{\lie a})$ is a Lie algebra, then there exists a unique bilinear map $[\cdot, \cdot] \colon \tlie a \times \tlie a \to \tlie a$ that satisfies
\[
[x \otimes f, \, y \otimes g] = [x, y]_{\lie a} \otimes fg,
\quad \textup{for all $x, y \in \lie a$, $f, g \in \C[t, t^{-1}]$},
\]
and endows $\tlie a$ with a structure of Lie algebra.  Notice that $\lie a \otimes \C$ is a subalgebra of $\lie a[t]$ isomorphic to $\lie a$, that $\lie a[t]$ is a subalgebra of $\tlie a$, and that $\lie b[t]$ (resp. $\tlie b$) is a  subalgebra of $\lie a[t]$ (resp. $\tlie a$) for all subalgebras $\lie b \subseteq \lie a$.

Fix an automorphism $\sigma \colon \lie a \to \lie a$.  Denote the order of $\sigma$ by $m$ and an $m$-th primitive root of unity by $\zeta$.  There exists a unique automorphism $\tilde\sigma \colon \tlie a \to \tlie a$ satisfying
\[
\tilde\sigma(x \otimes t^j) = \zeta^j \sigma(x) \otimes t^j
\qquad \textup{for all $x \in \lie a$ and $j \in \bb Z$}.
\]
Notice that $\tilde\sigma$ restricts to an automorphism of $\lie a[t]$ and that the order of $\tilde\sigma$ is also $m$.  Let $\tlie a^\sigma$ and $\lie a[t]^\sigma$ denote the subalgebras of $\tilde\sigma$-fixed points; explicitly
\[
\tlie a^\sigma = \bigoplus_{\epsilon = 0}^{m-1} \tlie a^\sigma_\epsilon, \quad
\tlie a^\sigma_\epsilon = \lie a_\epsilon \otimes t^{m-\epsilon}\C[t^m, t^{-m}],
\qquad \textup{and} \qquad
\lie a[t]^\sigma = \bigoplus_{\epsilon = 0}^{m-1} \lie a[t]^\sigma_\epsilon, \quad
\lie a[t]^\sigma_\epsilon = \lie a_\epsilon \otimes t^{m-\epsilon}\C[t^m].
\]
In particular, when $\sigma = {\rm id}_{\lie a}$, notice that $\lie a[t]^\sigma = \lie a[t]$ and $\tlie a^\sigma = \tlie a$.

Given a finite-dimensional complex simple Lie algebra $\lie g$ and a Dynkin diagram automorphism $\sigma \colon \lie g \to \lie g$, the associated twisted affine Kac-Moody algebra $\hlie g$ is the 2-dimensional extension of $\tlie g^\sigma$ given by $\hlie g:=\tlie g^\sigma \oplus \C c \oplus \C d$, and endowed with the unique Lie bracket that satisfies
\[
[x \otimes t^r, y \otimes t^s] = [x, y] \otimes t^{r+s} + r\ \delta_{r,-s}\ (x, y)\ c, \qquad
[c, \hlie g]=0 \qquad \textup{and} \qquad
[d, x\otimes t^r]=r\ x\otimes t^r
\]
for all $x \otimes t^r, y \otimes t^s \in \tlie g^\sigma$. Denote by $\hlie g' = [\hlie g, \hlie g]$ the derived subalgebra of $\hlie g$, notice that $\hlie g' = \tlie g^\sigma \oplus \C c$ and that we have a nonsplit short exact sequence of Lie algebras $0 \to \C c \to \hlie g' \to \tlie g^\sigma \to 0$.

Given a triangular decomposition $\lie g = \lie n^- \oplus \lie h \oplus \lie n^+$ and a Chevalley basis $\cal C = \{ x_i^\pm, h_\alpha \mid i \in I, \, \alpha \in R^+ \}$, denote $(\lie h \oplus \C c)$ by $\lie h'$, $(\lie g \otimes t^{\pm1} \C[t^{\pm1}])$ by $\lie g[t]_{\pm}$, and let
\[
\hlie h = \lie h \oplus \bb C c \oplus \bb C d, \qquad
\hlie n^{\pm} = \lie n^{\pm} \oplus \lie g [t]_{\pm}, \qquad
\hlie b^{\pm} = \hlie n^{\pm} \oplus \hlie h.
\]
A triangular decomposition of $\hlie g$ is given by $\hlie g = \hlie n^- \oplus\hlie h \oplus \hlie n^+$.  The root system, set of positive and simple roots associated to this triangular decomposition will be denoted by $\hat{R}$, $\hat R^+$ and $\hat{\Delta}$ respectively.  Notice that
\begin{equation}\label{e:rootlevel}
\alpha(c) = 0 \qquad \textup{for all}\quad \alpha\in\hat R.
\end{equation}
Now, let $\hat I=I\sqcup\{0\}$, $a_0 = 2$ for $\lie g \cong A_{2n}$, $a_0 = 1$ for $\lie g \not\cong A_{2n}$, $h_0 = \frac{m}{a_0}c - h_{\theta_1}$, and notice that $\{h_i \mid i\in\hat I\}\cup\{d\}$ is a basis of $\hlie h$. Identify $\lie h^*$ with the subspace $\{\lambda\in\hlie h^* \mid \lambda(c)=\lambda(d)=0\}$. Let $\delta$ be the unique linear map in $\hlie h^*$ such that $\delta(d)=1$ and $\delta(h_i)=0$ for all $i\in\hat I$. Define $\alpha_0 = \delta - \theta_1$. Then
\begin{gather*}
\hat\Delta=\{\alpha_i \mid i\in\hat I\}, \quad
\hat R^+ = R^+\cup\{\alpha+r\delta \mid r\in \bb Z_{>0}, \, \alpha \in R_r \cup\{0\}\}, \\
\hlie g_{\alpha + r \delta} = (\lie g_r)_\alpha \otimes t^r, \ 
\textup{if $r \in \bb Z, \, \alpha \in R_r$},
\quad \textup{and} \quad 
\hlie g_{r \delta} = \lie h \otimes t^r, \
\textup{if $r \in \bb Z \setminus \{0\}$}.
\end{gather*}
Given $\alpha\in R^+$, $r\in\bb Z_{\ge0}$, let
\begin{gather*}
x_{\pm\alpha+r\delta}^+ = x_{\alpha, -r}^\pm \otimes t^r, \qquad
x_{\pm\alpha+r\delta}^- = x_{\alpha, r}^\mp \otimes t^{-r}, \qquad
\hat r^\vee_\alpha = m r^\vee_\alpha / a_0, \\
h_{\pm\alpha+r\delta} =  [x_{\pm\alpha+r\delta}^+,x_{\pm\alpha-r\delta}^-]
= \begin{cases}
\delta_{r,1} \, h_{\frac \alpha 2, 0} + r \hat r^\vee_\alpha c, & \textup{if $\lie g \cong A_{2n}$ and $\alpha \in 2R_{\sh}$}, \\
\pm h_{\alpha,0} + r \hat r^\vee_\alpha c, & \textup{otherwise}.
\end{cases}
\end{gather*}

For each $i\in\hat I$, define $\Lambda_i$ to be the unique linear map in $\hlie h^*$ that satisfies $\Lambda_i(d)=0$ and $\Lambda_i(h_j)=\delta_{ij}$ for all $i,j\in\hat I$.  Set $\hat P = \mathbb Z\delta \oplus \left( \oplus_{i\in\hat I}\mathbb Z\Lambda_i \right)$, $\hat P^+ = \mathbb Z\delta \oplus \left( \oplus_{i\in\hat I} \mathbb N\Lambda_i \right)$, $\hat P'=\oplus_{i\in\hat I} \mathbb Z\Lambda_i$, and $\hat P'^+=\hat P'\cap\hat P^+$. Notice that
\[ \Lambda_0 (h) = 0 \quad \text{iff} \quad h \in \lie h \oplus \bb C d
\qquad \text{and} \qquad
\Lambda_i - \omega_i = \omega_i(h_\theta)\Lambda_0 \quad\text{for all}\ i \in I.\]
Hence, $\hat{P} = \bb Z \Lambda_0 \oplus P \oplus \bb Z \delta$. Given $\Lambda\in\hat P$, the number $\Lambda(c)$ is called the level of $\Lambda$. By \eqref{e:rootlevel}, the level of $\Lambda$ depends only on its class modulo the root lattice $\hat Q$. Set also $\hat Q^+ = \mathbb N \hat R^+$ and let
$\widehat{\mathcal W}$ denote the  affine Weyl group, which is generated by the simple reflections $s_i,i\in\hat I$.  Finally, observe that $\{\Lambda_0, \delta\}\cup\Delta$ is a basis of $\hlie h^\ast$.

Moreover, the set $\{ x_\alpha^\pm\otimes t^r, h_i\otimes t^r \mid \alpha\in R^+, \ i\in I, \ r\in\mathbb N\}$ forms a basis for $\lie g[t]$ and the set $\{ x_{\mu,-r}^\pm\otimes t^r, \, h_{i,-r}\otimes t^r \mid r \in \mathbb N, \, \mu \in R_{-r}, \, i\in I_0\} \setminus \{0\}$ forms a basis for $\lie g[t]^\sigma$.

%%%%%%%%%%%%%%%%%%%%%%%%%%%%%%%%%%%%%%%%%%%%%
\subsection{Integral forms and hyperalgebras}

Given a Lie algebra $\lie a$, let $U(\lie a)$ denote its universal enveloping algebra.  The Poincar\'e-Birkhoff-Witt (PBW) Theorem implies that there are isomorphisms of vector spaces
$$U(\lie g[t])\cong U(\lie n^-[t])\otimes U(\lie h[t])\otimes U(\lie n^+[t]) \quad\text{and}\quad U(\lie g[t]^\sigma)\cong U(\lie n^-[t]^\sigma)\otimes U(\lie h[t]^\sigma)\otimes U(\lie n^+[t]^\sigma).$$

Given $\alpha\in R^+$ and $k \in \mathbb Z_{\ge0}$, consider the following power series with coefficients in $U(\lie h[t])$
\begin{equation*}
\Lambda_{\alpha;k} (u)
= \sum_{r\ge 0}^{} \Lambda_{\alpha, r;k}u^r
= \exp\left(-\sum_{s>0}^{} \frac{h_{\alpha}\otimes t^{ sk}}{s}u^s\right).
\end{equation*}
For simplicity, denote $\Lambda_{\alpha, r; 0}$ by $\Lambda_{\alpha, r}$.  Similarly, consider the following power series with coefficients in $U(\lie h[t]^\sigma)$ (twisted analogues of the above elements). If either $\lie g$ is not of type $A_{2n}$ and $\mu \in R_{\sh}$, or $\lie g$ is of type $A_{2n}$ and $\mu\in R_0^+$, set
$$\Lambda_{\mu}^{\sigma}(u)  = \sum_{r=0}^\infty \Lambda_{\mu, r}^\sigma u^r =\text{exp} \left( - \sum_{k=1}^{\infty} \sum_{\epsilon=0}^{m-1}\frac{h_{\mu,\epsilon} \otimes t^{mk-\epsilon}}{mk-\epsilon} u^{mk-\epsilon}\right).$$
If $\lie g$ is not of type $A_{2n}$ and $\mu \in R_{\lng}$, set
$$\Lambda_{\mu}^{\sigma}(u)  = \sum_{r=0}^\infty \Lambda_{\mu,r}^\sigma u^r = \text{exp} \left( - \sum_{k=1}^{\infty} \frac{h_{\mu,0} \otimes t^{ mk}}{k} u^{k}\right). $$
For simplicity, denote $\Lambda_{\alpha_i,r}$ (resp. $\Lambda^\sigma_{\alpha_i,r}$) by $\Lambda_{i,r}$ (resp. $\Lambda^\sigma_{i,r}$).

One can easily check the following relation among twisted and non-twisted versions of the above elements. Given $\mu\in R_0^+$, let $\alpha\in O$ such that $\alpha|_{\lie h_0}=\mu$. Then:
\begin{equation}\label{e:tLvsL}
\Lambda_\mu^{\sigma,\pm}(u)=
\begin{cases} \prod_{j=0}^{m -1} \Lambda_{\sigma^j(\alpha)}^\pm(\zeta^{m-j}u), & \text{if}\quad  \Gamma_\alpha =m,  \\
\Lambda_{\alpha;m}^\pm(u),   & \text{if}\quad  \Gamma_\alpha = 1.  \end{cases}
\end{equation}

Given an order on the Chevalley basis of $\lie g[t]^\sigma$ and a PBW monomial with respect to this order, we construct an ordered monomial in the elements of the set
$$\cal M[t]^\sigma=\left\{(x^\pm_{\mu,-r}\otimes t^r)^{(k)}, \ \Lambda_{i, k}^{\sigma}, \ \binom{h_{i,0}}{k}  \mid  \mu\in\wt(\lie g_{-r})\cap Q_0^+\setminus\{0\}, i\in I_0, r\in \mathbb Z_{\ge0}, k\in \mathbb N\right\}\setminus\{0\},$$
using the correspondence $(x^\pm_{\mu,-r}\otimes t^r)^k \leftrightarrow (x^\pm_{\mu,-r}\otimes t^r)^{(k)}$, $h_{i,0}^k \leftrightarrow \binom{h_{i,0}}{k}$ and $(h_{i,-r}\otimes t^r)^k \leftrightarrow (\Lambda_{i,r}^\sigma)^k$ for $r\ne 0$.
Using the obvious similar correspondence we consider monomials in $U(\lie g)$ formed by elements of
$$\cal M=\left\{ (x_{\alpha}^\pm)^{(k)},\binom{h_i}{k}  \mid  \alpha\in R^+, i\in I, k\in\mathbb Z_{\ge0}\right\}$$
and in $U(\lie g[t])$ formed by elements of
$$\cal M[t]=\left\{(x^\pm_{\alpha}\otimes t^r)^{(k)}, \Lambda_{i, k},\binom{h_{i}}{k} \mid \alpha \in R^+, i\in I, k\in \mathbb N, r\in \mathbb Z_{\ge0}\right\}.$$
Notice that $\cal M$ can be naturally regarded as a subset of $\tilde{\cal M}$. The sets of ordered monomials thus obtained are basis of $U(\lie g[t]^\sigma)$, $U(\lie g)$, and $U(\lie g[t])$, respectively.

Let  $U_{\mathbb Z}(\lie g) \subseteq U(\lie g)$, $U_{\mathbb Z}(\lie g[t]) \subseteq U(\lie g[t])$, and $U_{\mathbb Z}(\lie g[t]^\sigma)\subseteq U(\lie g[t]^\sigma)$ be the $\mathbb Z$--subalgebras generated respectively by $\{(x_{\alpha}^\pm)^{(k)}  \mid  \alpha\in R^+, k\in\mathbb Z_{\ge0}\}$, $\{(x^\pm_{\alpha}\otimes t^r)^{(k)} \mid  \alpha \in R^+,r,k\in\mathbb Z_{\ge0}\}$, and $\{(x^\pm_{\mu,-r}\otimes t^r)^{(k)} \mid \mu\in\wt(\lie g_{-r})\cap Q_0^+\setminus\{0\}, r,k\in\mathbb Z_{\ge0}\}$.  It has been proven by Kostant \cite{kostant66} (in the $U(\lie g)$ case), Garland \cite{garland78} (in the $U(\tlie g)$ case), Mitzman \cite{mitzman85} and Prevost \cite{prevost92} (in the loop and twisted loop cases) that the subalgebras $U_{\mathbb Z}(\lie g[t]^\sigma)$, $U_{\mathbb Z}(\lie g)$, and $U_{\mathbb Z}(\lie g[t])$ are free $\mathbb Z$-modules and the sets of ordered monomials constructed from $\cal M[t]^\sigma$, $\cal M$, $\cal M[t]$ are $\mathbb Z$-basis of $U_{\mathbb Z}(\lie g[t]^\sigma)$, $U_{\mathbb Z}(\lie g)$, and  $U_{\mathbb Z}(\lie g[t])$, respectively.

In particular, it follows from this result that the natural map $\mathbb C\otimes_\mathbb ZU_\mathbb Z(\lie g[t]^\sigma)\to U(\lie g[t]^\sigma)$ is an isomorphism of vector spaces, i.e., $U_\mathbb Z(\lie g[t]^\sigma)$ is a an integral form of $U(\lie g[t]^\sigma)$, and similarly for $U_\mathbb Z(\lie g[t])$ and $U_\mathbb Z(\lie g)$. If $\lie a$ is a subalgebra of $\lie g$ preserved by $\sigma$, set
\begin{equation*}
U_\mathbb Z(\lie a[t]^\sigma) = U(\lie a[t]^\sigma)\cap U_\mathbb Z(\lie g[t]^\sigma)
\end{equation*}
and similarly define $U_\mathbb Z(\lie a)$ and $U_\mathbb Z(\lie a[t])$ for any $\lie a$ subalgebra of $\lie g$.  It also follows that, if $\lie a$ is either $\lie g_0$, $\lie n_0^\pm$, $\lie h_\epsilon$, $\lie h[t]^\sigma_\epsilon$,$\lie n^\pm[t]^\sigma$, $\lie n^\pm$, $\lie h$, $\lie n^\pm[t]$, or $\lie h[t]$, then $U_\mathbb Z(\lie a)$ is a free $\mathbb Z$-module spanned by monomials formed by elements of $\cal M[t]^\sigma$ belonging to $U(\lie a)$, and hence $\mathbb C\otimes_\mathbb Z U_\mathbb Z(\lie a) \cong U(\lie a)$.  Moreover, we have 
\begin{equation} \label{eq:Z.triang.dec}
U_\mathbb Z(\lie g[t]) = U_\mathbb Z(\lie n^-[t]) U_\mathbb Z(\lie h[t]) U_\mathbb Z(\lie n^+[t])
\quad \textup{and} \quad
U_\mathbb Z(\lie g[t]^\sigma) = U_\mathbb Z(\lie n^-[t]^\sigma) U_\mathbb Z(\lie h[t]^\sigma) U_\mathbb Z(\lie n^+[t]^\sigma).
\end{equation}

Given a field  $\mathbb F$, define the $\mathbb F$-hyperalgebra of $\lie a$ by
$$U_\mathbb F(\lie a) =  \mathbb F\otimes_{\mathbb Z}U_\mathbb Z(\lie a)$$
where $\lie a$ is any of the subspaces considered above. We will refer to $U_\mathbb F(\lie g[t]^\sigma)$ as the twisted hyper current algebra of $\lie g$ associated to $\sigma$ over $\mathbb F$.
Clearly, if the characteristic of $\mathbb F$ is zero, the algebra $U_\mathbb F(\lie g[t]^\sigma)$ is naturally isomorphic to $U(\lie g_\mathbb F[t]^\sigma)$ where $\lie g[t]_\mathbb F^\sigma=  \mathbb F\otimes_\mathbb Z\lie g[t]_\mathbb Z^\sigma$ and $\lie g[t]_\mathbb Z^\sigma$ is the $\mathbb Z$-span of the Chevalley basis of $\lie g[t]^\sigma$, and similarly for all algebras $\lie a$ we have considered. For fields of positive characteristic we just have an algebra homomorphism $U(\lie a_\mathbb F)\to U_\mathbb F(\lie a)$ which is neither injective nor surjective. We will keep denoting by $x$ the image of an element $x\in U_\mathbb Z(\lie a)$ in $U_\mathbb F(\lie a)$. Notice that we have
$$U_\mathbb F(\lie g[t])=U_\mathbb F(\lie n^-[t])U_\mathbb F(\lie h[t])U_\mathbb F(\lie n^+[t]) \quad  \text{ and } \quad  U_\mathbb F(\lie g[t]^\sigma)=U_\mathbb F(\lie n^-[t]^\sigma)U_\mathbb F(\lie h[t]^\sigma)U_\mathbb F(\lie n^+[t]^\sigma).$$

Notice that $U_\mathbb Z(\lie g)=U(\lie g)\cap U_\mathbb Z(\lie g[t])$, i.e., the integral form of $\lie g$ coincides with its intersection with he integral form of $U(\lie g[t])$ which allows us to regard $U_\mathbb Z(\lie g)$ as a $\mathbb Z$-subalgebra of $U_\mathbb Z(\lie g[t])$.

If $\lie g$ is of type $A_{2n}$ and the characteristic of $\mathbb F$ is 2, then $U_\mathbb F(\lie g_0)$ is not isomorphic to the usual hyperalgebra of $\lie g_0$ over $\mathbb F$,  constructed by using Kostant's integral form of $U(\lie g_0)$. However, if the characteristic of $\mathbb F$ is not 2, then $U_\mathbb F(\lie g_0)$ is isomorphic to the usual hyperalgebra of $\lie g_0$ over $\mathbb F$. Details can be found in \cite[Remark 1.5]{BM14} and \cite{bianchi12}.  This is the reason why we are not working with fields of characteristic 2 when $\lie g$ is of type $A_{2n}$.

%%%%%%%%%%%%%%%%%%%%%%%%%%%%%%%%%%%%%%%%%%%%%%%%%%%%
\subsection{A certain automorphism of hyperalgebras}

One can check that there exists a unique Lie algebra automorphism $\psi \colon \lie g \to \lie g$ satisfying $\psi(x_{\alpha_i}^\pm) = x_{\alpha_i}^\mp$ for all $i\in I$.  Moreover, $\psi$ admits a unique extension to a Lie algebra automorphism of $\tlie g$ (which we keep denoting by $\psi$) that satisfies $\psi(x\otimes f) = \psi(x)\otimes f$ for all $x\in\lie g, f\in\bb C[t, t^{-1}]$.  Notice that this automorphism restricts to an automorphism of $\lie g[t]$.  Also notice that $\psi \circ \sigma = \sigma \circ \psi$.  Thus $\psi$ restricts to automorphisms of $\tlie g^\sigma$ and $\lie g[t]^\sigma$.

Keep denoting by $\psi$ the extension of $\psi$ to an automorphism of $U(\tlie g)$.  In particular, we have
\begin{equation} \label{e:inv+-def}
\psi\left( \left( x_{\alpha,\epsilon}^\pm \otimes t^{rm - \epsilon} \right)^{(k)} \right)
= (-1)^{k({\rm ht}(\alpha)-1)} \left( x_{\alpha,\epsilon}^\mp \otimes t^{rm - \epsilon} \right)^{(k)}, \ 
\psi \binom{h_{\alpha,\epsilon}}{k}
= \binom{-h_{\alpha,\epsilon}}{k}, \ 
\psi \left( \Lambda_{\alpha, r}^\sigma \right)
= - \Lambda_{\alpha, r}^\sigma
\end{equation}
for all $\alpha\in R^+$, $\epsilon \in \{0, \dotsc, m-1\}$, $r \in \bb Z$, and $k \ge 0$.  It follows that $\psi$ restricts to an automorphism of $U_\bb Z(\lie a)$ for all $\lie a \in \{ \lie g_0, \, \lie h_0, \, \lie g[t], \, \lie h[t], \, \lie g[t]^\sigma, \, \lie h[t]^\sigma, \, \lie h[t]^\sigma_\pm\}$.  Moreover, notice that, for every $\mu \in P_0$, there exists a unique ring homomorphism $\mu \colon U_{\bb Z} (\lie h_0) \to \bb Z$ satisfying
\begin{equation} \label{eq:P.in.U_Fh}
\mu \binom{h_{i,0}}{k} = \binom{\mu(h_{i,0})}{k} \quad \textup{and} \quad
\mu(xy)=\mu(x)\mu(y)
\qquad \textup{for all $i\in I_0$, $k \ge 0$, $x,y \in U_\mathbb Z(\lie h_0)$}.
\end{equation}
Therefore,
\begin{equation}\label{e:inv+-w}
\mu\left(\psi \binom{h_{i,0}}{k} \right)
= \binom{-\mu(h_{i,0})}{k}
\qquad \textup{for all $i\in I_0$, $k>0$, $\mu\in P_0$}.
\end{equation}

%%%%%%%%%%%%%%%%%%%%%%%%%%%%%%%%%
\subsection{Technical identities}

The proof of the following lemma is very similar to that of \cite[Lema 2.2.8]{bianchi12}.

\begin{lemma} \label{isos} 
Let $\alpha\in R_0^+$.
	
\begin{enumerate}[(a), leftmargin=*]
\item Suppose $\lie g$ is not of type $A_{2n}$. 
\begin{enumerate}[(i)]
\item If $\alpha\in R_{\lng}$, then $\lie{sl}_{2,\alpha}[t^m] := {\rm span} \{x^\pm_{\alpha,0}\otimes t^{mk}, h_{\alpha,0}\otimes t^{mk}  \mid  k\in\mathbb Z_{\ge0}\}$ is a subalgebra of $\lie g[t]^\sigma$ isomorphic to $\lie{sl}_2[t]$. 

\item If $\alpha\in R_{\sh}$, then $\lie{sl}_{2,\alpha}[t] := {\rm span} \{x^\pm_{\alpha,\epsilon}\otimes t^{mk-\epsilon}, h_{\alpha,\epsilon}\otimes t^{mk-\epsilon}  \mid  k\in\mathbb Z_{\ge0},  \ 0 \leq \epsilon < m \}$ is a subalgebra of $\lie g[t]^\sigma$ isomorphic to $\lie{sl}_2[t]$.
\end{enumerate}
		
\item Suppose $\lie g$ is of type $A_{2n}$.
\begin{enumerate}[(i)]
\item If $\alpha\in R_{\lng}$, then $\lie{sl}_{2,\alpha}[t] := {\rm span} \{x^\pm_{\alpha,\epsilon}\otimes t^{mk-\epsilon}, h_{\alpha,\epsilon}\otimes t^{mk-\epsilon}  \mid  k\in\mathbb Z_{\ge0},  \ \epsilon =0,1 \}$ is a subalgebra of $\lie g[t]^\sigma$ isomorphic to $\lie{sl}_2 [t]$. 

\item If $\alpha \in R_{\sh}$, then ${\rm span}	\{x_{\alpha,\epsilon}^\pm\otimes t^{2k+\epsilon},x_{\alpha+\sigma(\alpha),1}^\pm\otimes t^{2k+1}, h_{\alpha,\epsilon}\otimes t^{2k+\epsilon}  \mid  k\in\mathbb Z_{\ge0}, \ \epsilon = 0,1 \}$ is a subalgebra of $\lie g[t]^\sigma$ isomorphic to $\lie{sl}_3[t]^{\tau}$, where $\tau$ is the nontrivial automorphism of $A_2$.
\end{enumerate}
\end{enumerate}
\end{lemma}

Given a monomial of the form $(x^\pm_{\mu_1,- r_1} \otimes t^{r_1})^{(k_1)} \cdots (x^\pm_{\mu_j,- r_j} \otimes t^{r_j})^{(k_j)}$, with $r_1,\dots,r_j \in \mathbb Z_{\ge0}$, a fixed choice of $\pm$, and $j \ge 1$, define its hyperdegree to be $k_1 + \cdots + k_j$. The following result was proved in \cite[Lemma 4.2.13]{mitzman85}.

\begin{lemma} \label{lem:comm1}
Let $r, s, j, k \in \bb Z_{\ge0}$, $\mu \in \wt(\lie g_{-r})$ and $\nu \in \wt(\lie g_{-s})$. Then $(x^\pm_{\mu,- r} \otimes t^r)^{(j)} (x^\pm_{\nu,-s} \otimes t^s)^{(k)}$ is in the $\bb Z$-submodule generated by $(x^\pm_{\nu,-s} \otimes t^s)^{(k)} (x^\pm_{\mu,- r} \otimes t^r)^{(j)}$ and by monomials of hyperdegree strictly smaller than $j+k$.
\end{lemma}

Given $\alpha\in R^+$ and $r,s\in \mathbb Z_{\ge0}$, consider the following power series with coefficients in $U(\lie g[t])$:
\[
X_{\alpha;(r,s)}(u) = \sum_{j=1}^{\infty} x_\alpha^- \otimes t^{r+(r+s)(j-1)}u^{j}.
\]
Next result was proven in \cite{garland78} in the loop algebras context and its proof remains valid for current algebras.

\begin{lemma} \label{l:garland} Let $\alpha\in R^+$, $r,s,\ell ,k\in \mathbb Z_{\ge0}$. We have
\[
(x_\alpha^+\otimes t^s)^{(\ell )}(x_\alpha^-\otimes t^r)^{(k)} = ({X_{\alpha;(r,s)}(u)}^{(k-\ell)})_{k} \ \mod U_\bb Z^+,
\]
where $U_\bb Z^+ := U_\bb Z (\lie n^- [t]) U_\bb Z (\lie h [t]_+)^0 + U_\bb Z(\lie g[t]) U_\bb Z(\lie n^+[t])^0$.
\end{lemma}

We now define certain power series with coefficients in $U(\lie g[t]^\sigma)$.  Let $r,s\in \mathbb Z_{\ge0}$. When either $\lie g$ is not of type $A_{2n}$ and $\mu \in R_{\sh}$, or $\lie g$ is of type $A_{2n}$ and $\mu\in R_0^+$, define
\[
X^{\sigma}_{\mu;(r,s)}(u) = \sum_{k=1}^\infty \left( x^-_{\mu,-(r+(r+s)(k-1))} \otimes t^{r+(r+s)(k-1)}\right) u^{k}.
\]
When $\lie g$ is not of type $A_{2n}$ and $\mu \in R_{\lng}$, define
\[
X^{\sigma}_{\mu;(r,s)}(u) = \sum_{k=1}^\infty \left( x^-_{\mu,0} \otimes t^{m(r+(r+s)(k-1)))} \right) u^{k}.
\]
For simplicity, we may denote 
$${U_\bb Z^+}^\sigma:=U_\bb Z (\lie n^- [t]^\sigma) U_\bb Z (\lie h [t]_+^\sigma)^0 + U_\mathbb Z(\lie{g}[t]^\sigma)U_\mathbb Z (\lie n^+ [t]^\sigma)^0,$$ 
$$X^\sigma_{\mu}:=X^\sigma_{\mu; (0,1)}, \quad \text{ and } \quad X^\sigma_{i;(r,s)}:=X^\sigma_{\alpha_i;(r,s)}.$$  
We now state a lemma whose proof follows directly from Lemma \ref{isos} (part (a)), Lemma \ref{l:garland} (part (b)) and \cite[Lemma 4.4.1(ii)]{mitzman85} (part (c)).

\begin{lemma} \label{basicreltw} 
Let $l,k,r,s\in\mathbb Z_{\ge0}$ and $0\le l\le k$.
	
\begin{enumerate}[(a), leftmargin=*]
\item If either $\lie g$ is not of type $A_{2n}$ and $\mu \in R_{\sh}$, or if $\lie g$ is of type $A_{2n}$ and $\mu \in R_{\lng}$, we have
\[
(x_{\mu,-s}^+ \otimes t^{s})^{(l)} (x_{\mu,-r}^- \otimes t^{r})^{(k)} 
= (-1)^l \left( X^{\sigma}_{\mu;(r,s)}(u)^{(k-l)} \right)_{k} 
\textup{ mod } {U_\bb Z^+}^\sigma.
\]
		
\item If $\lie g$ is not of type $A_{2n}$ and $\mu \in R_{\lng}$, we have
\[
(x_{\mu,0}^+\otimes t^{ms})^{(l)} (x_{\mu,0}^-\otimes t^{mr})^{(k)} 
= (-1)^l \left( X^{\sigma}_{\mu;(r,s)}(u)^{(k-l)}\right)_{k} 
\textup{ mod } {U_\bb Z^+}^\sigma.
\]
		
\item If $\lie g$ is of type $A_{2n}$ and $\mu\in R_{\sh}$, we have
\begin{enumerate}[(i)]
\item $(x_{\mu,0}^+ \otimes t^{ms})^{(k)} (x_{\mu,0}^- \otimes 1)^{(k+r)}  = (X(u)^{(r)})_{k+r} \textup{ mod } {U_\bb Z^+}^\sigma$ 	where $X(u) = \sum_{k=1}^{\infty} x_{\mu,0}\otimes t^{msj}u^{j}.$
			
\item $(x_{2\mu,1}^+\otimes t)^{(l)} (x_{2\mu,1}^-\otimes t)^{(k)} =(-1)^l \left( Y(u)^{(k-l)}\right)_{k} \textup{ mod } {U_\bb Z^+}^\sigma,$ where $Y(u) = \sum_{k=1}^{\infty} x_{2\mu,1}\otimes t^{2k-1}u^k.$
			
\item $(x_{2\mu,1}^+\otimes t)^{(k)} (x_{\mu,0}^-\otimes 1)^{(2k+r)} =\displaystyle{\sum_{\substack{k_1,k_2\\ k_1+2k_2=r}}\left( X^{\sigma}_{\mu}(u)^{(k_1)} Z(u)^{(k_2)}\right)_{k+k_1} }  \textup{ mod } {U_\bb Z^+}^{\sigma}$, where $Z(u) = \sum_{k=1}^{\infty} x_{2\mu,1}\otimes t^{2k-1}u^{2k-1}.$\qedhere
			 
\end{enumerate}
\end{enumerate}
\end{lemma}

%%%%%%%%%%%%%%%%%%%%%%%%%%%%%%%%%%
\section{Preliminaries on modules}

%%%%%%%%%%%%%%%%%%%%%%%%%%%%%%%%%%%%%%%%%%%%%%%%%
\subsection{Integrable modules for hyperalgebras} \label{sec:rev.g}

Let ${\mathbb F}$ be either a field or $\mathbb Z$.  In this subsection we will review the representation theory of $U_{\mathbb F} (\lie g_0)$.

Let $V$ be a $U_{\mathbb F} (\lie g_0)$-module.  For each $\mu \in U_{\mathbb F} (\lie h_0)^*$, denote the weight space of weight $\mu$ by $V_\mu = \{ v \in V \mid hv = \mu(h)v \ \textup{for all $h \in U_{\mathbb F} (\lie h_0)$}\}$ .  A nonzero vector $v \in V_\mu$ is said to be a weight vector.  If $V_\mu\ne 0$, then $\mu$ is said to be a weight of $V$.  Let $\wt_0(V) = \{\mu\in U_{\mathbb F} (\lie h_0)^* \mid V_\mu\ne 0\}$ denote the set of weights of $V$.  If $V = \bigoplus_{\mu \in \wt_0(V)} V_\mu$, then $V$ is said to be a weight module.

Using the inclusion $P_0 \hookrightarrow U_{\mathbb F} (\lie h_0)^*$ induced by \eqref{eq:P.in.U_Fh}, define a partial order $\le$ on $U_{\mathbb F} (\lie h_0)^*$ given by: $\mu\le\lambda$ if $\lambda-\mu\in Q_0^+$.  Notice that
$(x_{\alpha, 0}^\pm)^{(k)} V_\mu \subseteq V_{\mu\pm k\alpha} \qquad
\textup{for all} \ \alpha\in R_0^+, \ k > 0, \ \mu\in U_{\mathbb F} (\lie h_0)^*$.
Now, if $v \in V$ is a weight vector such that $(x_{\alpha, 0}^+)^{(k)} v = 0$ for all $\alpha\in R_0^+$ and $k>0$, then $v$ is said to be a highest-weight vector.  If $V$ is generated by a highest-weight vector, then it is said to be a highest-weight module.  If, for each $v \in V$, there exists $m>0$ such that $(x_{\alpha,0}^\pm)^{(k)}v=0$ for all $\alpha \in R_0^+$ and $k>m$, then $V$ is said to be integrable.

If $V$ is a weight-module whose weight spaces are finitely-generated $\mathbb F$-modules, its character is defined to be the function $\ch(V) \colon U_{\mathbb F} (\lie h_0)^* \to \mathbb Z$ given by $\ch(V)(\mu) = {\rm rank}_{\mathbb F} (V_\mu)$.  Thus, $\ch(V)$ can be regarded as an element of the group ring $\mathbb Z [U_{\mathbb F} (\lie h_0)^*]$.  Recall that \eqref{eq:P.in.U_Fh} induces an inclusion $P_0 \hookrightarrow U_{\mathbb F} (\lie h_0)^*$.  Thus, the group ring $\mathbb Z[P_0]$ can be regarded as a subring of $\mathbb Z[U_{\mathbb F} (\lie h_0)^*]$.  Denote by $e^\mu$ the element of $\bb Z [U_{\mathbb F} (\lie h_0)^*]$ corresponding to $\mu \in U_{\mathbb F} (\lie h_0)^*$.  There exists a unique action of $\cal W_0$ on $\mathbb Z[P_0]$ (by ring automorphisms) that satisfy $w \cdot e^\mu = e^{w\mu}$ for all $w \in \cal W_0$, $\mu \in P_0$.

\begin{theorem} \label{t:rh}
Let $V$ be a $U_{\mathbb F} (\lie g_0)$-module.

\begin{enumerate}[(a), leftmargin=*]
\item  \label{t:rh.a}
If $V$ is a finitely-generated $\mathbb F$-module, then $V$ is a weight module and $\wt_0 (V) \subseteq P_0$.

\item \label{t:rh.WinvZ}
If $V$ is an integrable weight module, then there are isomorphisms of ${\mathbb F}$-modules $V_\mu \cong V_{w\mu}$ for all $w \in \cal W_0$, $\mu \in P_0$.  Thus, $\ch(V) \in \mathbb Z[P_0]^{\cal W_0}$.

\item If $V$ is a highest-weight module of highest weight $\lambda$, then $V_{\lambda}$ is a free $\bb F$-module of rank $1$ and $V_{\mu}\ne 0$ only if $\mu\le \lambda$. Moreover, $V$  has a unique maximal proper submodule and, hence, also a unique irreducible quotient. In particular, $V$ is indecomposable.

\item \label{t:rh.d}
For each $\lambda\in P_0^+$, the $U_{\mathbb F} (\lie g_0)$-module $W_{\mathbb F} (\lambda)$ given as the quotient of $U_{\mathbb F} (\lie g_0)$ by the left ideal generated by
\begin{equation*}
x_{\alpha,0}^+, \qquad 
h-\lambda(h), \qquad
(x_{\alpha,0}^-)^{(k)}, \qquad 
\textup{for all} \ \alpha\in R_0^+, \ h \in U_{\mathbb F} (\lie h_0), \ k>\lambda(h_{\alpha,0}),
\end{equation*}
is a nonzero, free ${\mathbb F}$-module of finite rank.  Moreover, every free ${\mathbb F}$-module of finite rank that is a highest-weight module of highest weight $\lambda$ is a quotient of $W_{\mathbb F} (\lambda)$.

\item If ${\mathbb F}$ is a field and $V$ is a finite-dimensional irreducible $U_{\mathbb F} (\lie g_0)$-module, then there exists a unique $\lambda\in P_0^+$ such that $V$ is isomorphic to the irreducible quotient $V_{\mathbb F} (\lambda)$ of $W_{\mathbb F} (\lambda)$. If the characteristic of ${\mathbb F}$ is zero, then $W_{\mathbb F}(\lambda)$ is irreducible.

\item \label{t:chWg}
For each $\lambda\in P_0^+$, the character of $W_{\mathbb F} (\lambda)$ is given by the Weyl character formula.  In particular, $\mu \in \wt_0 (W_{\mathbb F} (\lambda))$ if, and only if, $w\mu \le \lambda$ for all $w \in \cal W_0$.
\end{enumerate}
\end{theorem}

\begin{proof}
If ${\mathbb F} = \bb Z$, then parts~\ref{t:rh.a}, \ref{t:rh.d}, \ref{t:chWg} are proved in \cite[Theorem~2.3.6]{macedo13} and part~\ref{t:rh.WinvZ} is proved in \cite[Proposition 2.3.5]{macedo13}.

If ${\mathbb F}$ is a field, then $U_{\mathbb F} (\lie g_0)$ is the algebra of distributions of an algebraic group of the same Lie type as $\lie g_0$.  Thus, this result is proved in \cite[Part II]{jan03}.  In the particular case where ${\mathbb F}$ is a field of characteristic zero, $U_{\mathbb F} (\lie g_0) \cong U((\lie g_0)_{\mathbb F})$, where $(\lie g_0)_{\mathbb F} = {\mathbb F} \otimes_\bb Z (\lie g_0)_\bb Z$ and $(\lie g_0)_\bb Z$ is the $\bb Z$-span of the Chevalley basis given in Section~\ref{s:chevalley}.  In this case, this result is also proved in \cite{humphreys90}.  In the particular case where ${\mathbb F}$ is a field of positive characteristic, this result is also proved in \cite[Section 2]{JM1}.
\end{proof}

%%%%%%%%%%%%%%%%%%%%%%%%%%%%%%%%%%%%%%%%%%%%%%%%%%%%%%%%%%%%%%%%%%%%%%%%%%%%%%
\subsection{Demazure and local Weyl modules for twisted current hyperalgebras}\label{s:wdm}

Throughout this subsection, let $\bb F$ be either a field or $\bb Z$, let $\lie g$ be a finite-dimensional simple $\bb C$-Lie algebra with a triangular decomposition $\lie g = \lie n^- \oplus \lie h \oplus \lie n^+$, and let $\sigma \colon \lie g \to \lie g$ be a Dynkin diagram automorphism.

\begin{definition}\label{d:weyl}
Given $\lambda \in P^+_0$, let the twisted local graded Weyl module $W_\bb F^{c, \sigma} (\lambda)$ be the $U_\bb F (\lie g[t]^\sigma)$-module given as a quotient of $U_\bb F (\lie g[t]^\sigma)$ by the left ideal generated by
\begin{equation*}
U_\bb F (\lie n^+[t]^\sigma)^0, \ \,
U_\bb F (\lie h[t]_+^\sigma)^0, \ \,
h - \lambda(h), \ \,
(x_{\alpha,0}^-)^{(k)}, \quad \,
\textup{for all $h \in U_\bb F (\lie h_0)$, $\alpha \in R_0^+$, $k > \lambda( h_{\alpha,0} )$}.
\end{equation*}
Given $\ell \geq 0$ and $\lambda \in P^+_0$, let the twisted Demazure module $D_\bb F^\sigma (\ell, \lambda)$ be the $U_\bb F (\lie g[t]^\sigma)$-module given as a quotient of $U_\bb F (\lie g[t]^\sigma)$ by the left ideal generated by
\begin{equation} \label{eq:defn.rel.dem}
U_\bb F (\lie n^+[t]^\sigma)^0, \qquad
U_\bb F (\lie h[t]_+^\sigma)^0, \qquad
h - \lambda(h), \quad
(x_{\alpha,-s}^- \otimes t^s)^{(k)},
\end{equation}
for all $h \in U_\bb F (\lie h_0)$, $\alpha \in R_0^+$, $s \in \bb Z_{\ge0}$, $k > \max\{ 0, \delta_{s,1}\lambda(h_{\frac{\alpha}{2},0}) - \hat r_\alpha^\vee \ell s\}$, in case $ \lie g \cong A_{2n}$, $\alpha \in 2R_{\sh}$, and $k > \max\{ 0, \lambda(h_{\alpha,0}) - \hat r_\alpha^\vee \ell s\}$ in every other case.
\end{definition}

Notice that $W^{c,\sigma}_\bb F (\lambda)$, $D_\bb F^\sigma (\ell, \lambda)$ are weight modules, and that $D_\bb F^\sigma (\ell, \lambda)$ is a quotient of $W^{c,\sigma}_\bb F (\lambda)$.  Moreover, the ideal defining $W^{c,\sigma}_\bb F (\lambda)$ (resp. $D_\bb F^\sigma (\ell, \lambda)$) is the image of the ideal defining $W^{c,\sigma}_\bb Z (\lambda)$ (resp. $D_\bb Z^\sigma (\ell, \lambda)$) in $U_\bb F(\lie g[t]^\sigma)$.  Thus, we have the following isomorphisms of $U_\bb F (\lie g [t]^\sigma)$-modules:
\[
W^{c,\sigma}_\bb F (\lambda) \cong \bb F \otimes_\bb Z W^{c,\sigma}_\bb Z (\lambda)
\qquad \textup{and} \qquad
D_\bb F^\sigma (\ell, \lambda) \cong \bb F \otimes_\bb Z D_\bb Z^\sigma (\ell, \lambda).
\]

%%%%%%%%%%%%%%%%%%%%%%%%%%%%%%%%%%%%%%
\section{Preliminary results} \label{prereq}

\subsection{Finite-dimensionality of Weyl modules}

The following lemma will be used in Proposition~\ref{prop:WZfg}.

\begin{lemma} \label{l:nilp} \
\begin{enumerate}[(a), leftmargin=*]
\item \label{l:nilsl2t}
If $V$ is a finite-dimensional $U_\mathbb F(\lie g[t])$-module, $\lambda \in P^+$ and $v \in V_\lambda$ is such that
\[
U_\bb F (\lie n^+ [t])^0 v = U_\bb F (\lie h[t]_+)^0 v = 0,
\]
then $(x^-_\alpha \otimes t^s)^{(k)}v = 0$ for all $\alpha \in R^+, s \ge \lambda (h_\alpha)$ and $k \in \bb Z_{>0}$.
		
\item \label{l:nilsl2tsigma}
If $V$ is a finite-dimensional $U_\mathbb F(\lie g[t]^\sigma)$-module, $\lambda \in P_0^+$ and $v \in V_\lambda$ is such that
\[
U_\mathbb F(\lie n^+[t]^\sigma)^0 v = U_\bb F (\lie h[t]_+^\sigma)^0 v = 0, 
\]
then $(x^-_{\mu,-s} \otimes t^s)^{(k)}v = 0$ for all $x^-_{\mu,-s} \in C^\sigma (O), s \ge d_\mu \lambda(h_{\mu,0})$ and $k \in \bb Z_{>0}$, where $d_\mu=2$ if $\lie g$ is of type $A_{2n}$ and $\mu \in R_{\sh}$, and $d_\mu=1$ otherwise. 
\end{enumerate}
\end{lemma}

\begin{proof}
The proof of part~\ref{l:nilsl2t} is similar to that of \cite[Proposition 3.1]{JM1}. The proof of part~\ref{l:nilsl2tsigma} is similar to that of \cite[Proposition 3.2]{BM14}.
\end{proof}

\begin{lemma} \label{lem:WZint}
For any $\lambda \in P^+_0$, $W_\bb Z^{c, \sigma} (\lambda)$ is an integrable $U_\bb Z (\lie g_0)$-module.
\end{lemma}

\begin{proof}
This proof is similar to the proof of \cite[Proposition~3.1.1]{JM1}.
\end{proof}

\begin{proposition} \label{prop:WZfg}
For any $\lambda \in P^+_0$, $W_\bb Z^{c, \sigma} (\lambda)$ is a finitely generated abelian group.
\end{proposition}

\begin{proof}
Let $\Xi$ be the set of functions $\phi \colon \mathbb Z_{>0} \to Q^+ \times \bb Z_{\ge 0} \times \bb Z_{\ge 0}$, $\phi(j)=(\beta_j,r_j,k_j)$, such that: $\beta_j \in R^+_{-r_j}$ for all $j \in \bb Z_{>0}$, and $k_j=0$ for all $j$ sufficiently large.  Let $v$ denote the image of $1 \in U_\bb Z (\lie g[t]^\sigma)$ in $W_\bb Z^{c,\sigma} (\lambda)$, and recall that $W_\bb Z^{c,\sigma} (\lambda)$ is generated (as an abelian group) by elements of the form
\[
v_\phi
= (x_{\beta_1,-r_1}^- \otimes t^{r_1})^{(k_1)} \cdots (x_{\beta_n,-r_n}^- \otimes t^{r_n})^{(k_n)} v,
\quad \phi \in \Xi.
\]
Thus, the subgroup of $W_\bb Z^{c,\sigma} (\lambda)$ generated by elements $v_\phi$ with a fixed $\sum_{j>0} r_j$ and a fixed $\lie h_0$-weight ($\lambda - \sum_{j>0} k_j \beta_j$) is finitely generated.

By Lemma~\ref{lem:WZint}, we know that $W^{c,\sigma}_\bb Z (\lambda)$ is an integrable $U_\bb Z(\lie g_0)$-module.  Thus, Theorem~\ref{t:rh}\ref{t:rh.WinvZ} implies that the set $\wt_0 \left( W_\bb Z^{c,\sigma} (\lambda) \right)$ is $\cal W_0$-invariant.  Since the $\lie h_0$-weights of $W_\bb Z^{c,\sigma} (\lambda)$ are bounded above by $\lambda$, it follows that the set of weights of $W_\bb Z^{c,\sigma} (\lambda)$ is contained in that of $W_\bb Z (\lambda)$.  Since the set of weights of $W_\bb Z (\lambda)$ is finite, it follows that the set of weights of $W_\bb Z^{c,\sigma} (\lambda)$ is also finite.  Hence, the subgroup of $W_\bb Z^{c,\sigma} (\lambda)$ generated by elements $v_\phi$ with a fixed $\sum_{j>0} r_j$ is finitely generated.

In order to show that $W^{c,\sigma}_\bb Z(\lambda)$ is finitely generated, we will restrict the possibilities of $\sum_{j>0} r_j$ to finitely many distinct ones.  Define the exponent of $\phi \in \Xi$ to be $e(\phi)= \max\{k_j \mid j>0\}$, observe that $0 \leq e(\phi)\le d(\phi)$, and let
\begin{gather*}
\Xi_{d,e} = \{\phi\in \Xi \mid d(\phi)=d , e(\phi)=e\},
\qquad
\Xi_d= \bigcup_{0 \le e \le d} \Xi_{d,e} , \\
\Xi' = \{ \phi \in \Xi \mid \phi(j) = (\beta_j,r_j,k_j), \ r_j < d_{\beta_j}\lambda(h_{\beta_j,0}) \ \textup{ for all $j \in \bb Z_{>0}$} \} , \\[1em]
W' = \sum_{\phi \in \Xi'} \bb Z v_\phi \subseteq W_\bb Z^{c,\sigma} (\lambda), \qquad 
W_{d} = \sum_{\phi \in \Xi_d \cap \Xi'} \bb Z v_\phi \subseteq W', \\[1em]
\textup{where} \quad d_{\beta_j} = \begin{cases} 2, & \textup{if $\lie g \cong A_{2n}$ and $\beta_j \in R_{\sh}$}, \\ 1, & \textup{else}. \end{cases}
\end{gather*}
In order to finish the proof, we will show that $W' = W_\bb Z^{c,\sigma} (\lambda)$.

In fact, we will use induction on $d$ and $e$ in order to show that $v_\phi \in W'$ for all $\phi \in \Xi_{d,e}$, $d, e \ge 0$. If $d(\phi) = e(\phi) = 0$, then $v_\phi = v \in W'$. Assume that $d(\phi)>0$ and that the statement holds for all $\phi' \in (\cup_{d<d(\phi)}\Xi_{d}) \cup (\cup_{e<e(\phi)}\Xi_{d(\phi),e})$. The proof splits into two cases, according to whether $e(\phi)=d(\phi)$ or not.

Suppose $e(\phi)<d(\phi)$, in which case $v_\phi = (x_{\beta_1, -r_1}^- \otimes t^{r_1})^{(k_1)} \cdots (x_{\beta_n,-r_n}^- \otimes t^{r_n})^{(k_n)} v$ for some $n>1$ and $k_1, \ldots, k_n \neq 0$. By induction hypothesis, $(x_{\beta_2, -r_2}^- \otimes t^{r_2})^{(k_2)} \cdots (x_{\beta_n,-r_n}^- \otimes t^{r_n})^{(k_n)} v \in W'$. Thus, without loss of generality, we will assume that $r_j < d_{\beta_j}\lambda(h_{\beta_j,0})$ for all $j \in \{2, \dotsc, n\}$. Using Lemma~\ref{lem:comm1} repeatedly to commute $(x_{\beta_1,-r_1}^- \otimes t^{r_1})^{(k_1)}$ with $(x_{\beta_2,-r_2}^- \otimes t^{r_2})^{(k_2)}, \ldots, (x_{\beta_n,-r_n}^- \otimes t^{r_n})^{(k_n)}$, and induction hypothesis on the terms of hyperdegree strictly less than $d(\phi)$, we obtain that
\[
v_\phi + W' 
= (x_{\beta_2, -r_2}^- \otimes t^{r_2})^{(k_2)} \cdots (x_{\beta_n, -r_n}^- \otimes t^{r_n})^{(k_n)} (x_{\beta_1, -r_1}^- \otimes t^{r_1})^{(k_1)} v + W'.
\]
By induction hypothesis, $(x_{\beta_1, -r_1}^- \otimes t^{r_1})^{(k_1)} v \in W'$.  Since $r_2 < d_{\beta_2}\lambda (h_{\beta_2, 0}), \dots, r_n < d_{\beta_j}\lambda (h_{\beta_n, 0})$, it follows that
\[
(x_{\beta_2, -r_2}^- \otimes t^{r_2})^{(k_2)} \cdots (x_{\beta_n, -r_n}^- \otimes t^{r_n})^{(k_n)} (x_{\beta_1, -r_1}^- \otimes t^{r_1})^{(k_1)} v \in W'.
\]
Hence, $v_\phi \in W'$.

Now, suppose $e(\phi)=d(\phi)$, in which case $v_\phi = (x_{\beta,-r}^- \otimes t^r)^{(d)} v$.  Also suppose that $r \geq d_{\beta}\lambda(h_{\beta,0})$. We will split the rest of this proof into four cases:

\noindent
\textbf{Case 1:} Assume $\lie g \not\cong A_{2n}$ and $\beta \in R_{\sh}$ or $\lie g \cong A_{2n}$ and $\beta \in R_{\lng}$. Our goal is to show that $( x_{\beta, -r}^- \otimes t^r )^{(d)} v \in W'$ for all $d,r\in \mathbb N$ with $r \ge d_\beta \lambda(h_{\beta, 0})$. Replacing $k = l + d, r = 0, s = 1$ in Lemma \ref{basicreltw}(a) we get
\[
(x_{\beta,-1}^+ \otimes t)^{(l)} (x_{\beta,0}^- \otimes 1)^{(l+d)} v
= (-1)^l \left( \left(\sum_{r \ge 0} (x_{\beta, -r}^- \otimes t^r) u^{r+1}\right)^{(d)} \right)_{l+d} v.
\]
Since $(x_{\beta,0}^- \otimes 1)^{(\kappa)} v = 0$ for all $\kappa > \lambda(h_{\beta, 0})$, it follows that this last equality is zero
\iffalse
\[
(x_{\beta,-1}^+ \otimes t)^{(l)} (x_{\beta,0}^- \otimes 1)^{(l+d)} v =  \big( \big( \sum_{r \ge 0} (x_{\beta, -r}^- \otimes t^r) u^{r+1} \big)^{(d)} \big)_{l+d} v = 0,
\]
\fi
for all $l \ge \lambda(h_{\beta, 0})$. Thus
\begin{equation} \label{eq:prop.fin.dim.A}
\sum_{\substack{r_1n_1 + \dots + r_d n_d= l \\ n_1 + \dots + n_d= d }} (x_{\beta, -r_1}^- \otimes t^{r_1})^{(n_1)} \ldots (x_{\beta, -r_d}^- \otimes t^{r_d})^{(n_d)} v = 0.
\end{equation}
Set $l = d(\lambda(h_{\beta,0})+l'), l' \ge 0$. Since, from the case $e(\phi)<d(\phi)$,
\[
(x_{\beta, -r_1}^- \otimes t^{r_1})^{(d_1)} \ldots (x_{\beta, -r_{d'}}^- \otimes t^{r_{d'}})^{(d_{d'})} v \in W'
\]
for all $d'>1$, it follows, from Equation \eqref{eq:prop.fin.dim.A}, that $(x_{\beta, -r}^- \otimes t^{\lambda(h_{\beta,0})+l'})^{(d)} v \in W'$ for all $l' \ge 0$.

\noindent
\textbf{Case 2:} Assume $\lie g \not\cong A_{2n}$ and $\beta \in R_{\lng}$. Here our goal is to show that $( x_{\beta, 0 }^- \otimes t^{r} )^{(d)} v \in W'$ for all $d\in\mathbb N$ and $r\in m\mathbb N$ with $r \ge \lambda(h_{\beta, 0})$. Replacing $k = l + d, r = 0, s = 1$ in Lemma \ref{basicreltw}(b) we get
\[
(x_{\beta,0}^+ \otimes t^m)^{(l)} (x_{\beta,0}^- \otimes 1)^{(l+d)} v
= (-1)^l \left(\left(\sum_{r \ge 0}x_{\beta, 0}^- \otimes t^{mr} u^{r+1}\right)^{(d)}\right)_{l+d} v.
\]
Since $(x_{\beta,0}^- \otimes 1)^{(\kappa)} v = 0$ for all $\kappa > \lambda(h_{\beta, 0})$, it follows that this last equality is zero 
\iffalse
\[
(x_{\beta,0}^+ \otimes t^m)^{(l)} (x_{\beta,0}^- \otimes 1)^{(l+d)} v = \big( \big( \sum_{r \ge 0} (x_{\beta, 0}^- \otimes t^{mr}) u^{r+1} \big)^{(d)} \big)_{l+d} v = 0,
\]
\fi
for all $l \ge \lambda(h_{\beta, 0})$. Thus
\begin{equation}  
\sum_{\substack{r_1n_1 + \dots + r_d n_d= l \\ n_1 + \dots + n_d= d }}(x_{\beta, 0}^- \otimes t^{mr_1})^{(n_1)} \ldots (x_{\beta, 0}^- \otimes t^{mr_d})^{(n_d)} v=0,
\end{equation}
and the conclusion follows as in the previous case.

\noindent
\textbf{Case 3:} Assume $\lie g \cong A_{2n}$ and $\beta \in 2R_{\sh}$.  For this case our goal is to show that $(x_{\beta, -r}^- \otimes t^r )^{(d)} v \in W'$ for all $d,r\in\mathbb N$ with $r \ge d_{\beta}\lambda(h_{\beta, 0})$ and $r$ an odd number. It is achievable proceeding in a similar way to the previous cases by using Lemma \ref{basicreltw}(c)(ii). 

 \noindent
 \textbf{Case 4:} Assume $\lie g \cong A_{2n}$ and $\beta \in R_{\sh}$. Our goal is to show that $(x_{\beta, -r}^- \otimes t^r )^{(d)} v \in W'$ for all $d,r\in\mathbb N$ with $r \ge d_{\beta}\lambda(h_{\beta, 0})$. The procedure with $r$ even is done by using Lemma \ref{basicreltw}(c)(i) also in a similar way to the other cases. Finally, the case with $r$ odd follows by Lemma \ref{basicreltw}(c)(iii) and by using the conclusion of Case 3.
\qedhere

\end{proof}

\begin{corollary} \label{cor:Weyl.Dem.fd}
For any $\lambda \in P_0^+$ and $\ell > 0$, $W^{c,\sigma}_\bb F (\lambda)$ and $D_\bb F^\sigma (\ell, \lambda)$ are finite-dimensional.
\end{corollary}

\begin{proof}
Recall that $W^{c,\sigma}_\bb F (\lambda) \cong \bb F \otimes_\bb Z W_\bb Z^{c, \sigma} (\lambda)$. Since $W^{c, \sigma}_\bb Z (\lambda)$ is finitely generated, the dimension of $W^{c, \sigma}_\bb F (\lambda)$ is at most the number of generators of $W^{c, \sigma}_\bb Z ( \lambda)$.  Now recall that $D_\bb Z^\sigma (\ell, \lambda)$ is a quotient of $W_\bb Z^{c, \sigma} (\lambda)$, for all $\ell > 0$ and $\lambda \in P_0^+$. Since $W_\bb Z^{c, \sigma} (\lambda)$ is finitely generated, it follows that $D_\bb Z^\sigma (\ell, \lambda)$ is also finitely generated.  Finally, recall that $D^{\sigma}_\bb F (\ell, \lambda) \cong \bb F \otimes_\bb Z D_\bb Z^{\sigma} (\ell, \lambda)$, and, since $D_\bb Z^\sigma (\ell, \lambda)$ is finitely generated, it follows that $D^\sigma_\bb F (\ell, \lambda)$ is also finite-dimensional.
\end{proof}

%%%%%%%%%%%%%%%%%%%%%%%%%%%%%%%%%%%%%%%%%%%%%%%%%%%%%%%%%%%%%%%%%%%%%%%%%%
\subsection{The category of $\bb Z$-graded finite-dimensional $U_\bb F (\lie g[t]^\sigma)$-modules}

Let $\cal G_\bb F^\sigma$ be the category of $\bb Z$-graded finite-dimensional $U_\bb F (\lie g[t]^\sigma)$-modules. Given a module $V$ in $\cal G_\bb F^\sigma$, let $V[r]$ denote its $r$-th graded piece, and given $s \in \bb Z$, let $\tau_s V$ be the module in $\cal G_\bb F^\sigma$ satisfying $(\tau_s V) [r] = V[r-s]$ for all $r \in \bb Z$. For each finite-dimensional $U_\bb F (\lie g_0)$-module $V$, let $\ev_0 V$ be the module in $\cal G_\bb F^\sigma$ obtained by inflating the action of $U_\bb F (\lie g_0)$ to $U_\bb F (\lie g[t]^\sigma)$ by setting $U_\bb F (\lie g [t]^\sigma_+) V = 0$. For each $r \in \bb Z$, set $\ev_r = \tau_r \circ \ev_0$, and for each $\lambda \in P_0^+$, set $V_\bb F (\lambda,r) = \ev_r V_\bb F(\lambda)$.

\begin{theorem}\ 
\begin{enumerate}[(a), leftmargin=*]
\item If $V$ is a simple module in $\cal G_\bb F^\sigma$, then $V$ is isomorphic to $V_\mathbb F(\lambda,r)$ for a unique $(\lambda,r)\in P_0^+\times\bb Z$.
\item For every $\lambda \in P^+_0$, $W^{c,\sigma}_\mathbb F(\lambda)$ is in $\cal G_\bb F^\sigma$.
\item Let $\lambda \in P^+_0$. If $V$ is a module in $\cal G_\bb F^\sigma$ generated by a nonzero vector $v$ satisfying
\begin{equation} \label{eq:thm.3.4.c}
U_\bb F (\lie n^+[t]^\sigma)^0 v = U_\bb F(\lie h[t]_+^\sigma)^0 v=0
\quad \textup{and} \quad
hv = \lambda(h)v
\quad 
\textup{for all }
h \in U_\bb F(\lie h_0),
\end{equation}
then $V$ is a quotient of $W^{c,\sigma}_\bb F(\lambda)$.
\end{enumerate}
\end{theorem}

\begin{proof}
To prove part (a), suppose $V[r], V[s] \neq 0$ for some $s < r \in \bb Z$. In that case, $\oplus_{k\geq r} V [k]$ would be a proper submodule of $V$, contradicting the fact that it is simple. Thus there must exist a unique $r \in \bb Z$ such that $V[r] \neq 0$. Since $uv \in V[s]$, $s>r$, for all $u \in U_\bb F (\lie g[t]^\sigma_+)^0$ and $v \in V = V[r]$, it follows that $U_\bb F (\lie g[t]^\sigma_+)^0 V = 0$.  Thus $V$ is in fact the inflation of a simple $U_\bb F (\lie g_0)$-module; that is, $V \cong V_\bb F (\lambda,r)$ for some $\lambda \in P_0^+$ and $r\in \bb Z$.

Part (b) follows directly from Corollary \ref{cor:Weyl.Dem.fd}.

To prove part (c), observe that the $U_\bb F (\lie g_0 [t])$-submodule $V' = U_\bb F (\lie g_0 [t]) v \subseteq V$ is a graded, finite-dimensional, highest-weight module of highest weight $\lambda$. Thus $V'$ is a quotient of the graded Weyl module for $U_\bb F (\lie g_0[t])$ of highest weight $\lambda$. Using Lemma \ref{isos}, the statement follows by comparing the defining relations of $W^{c,\sigma}_\bb F(\lambda)$ with those in \eqref{eq:thm.3.4.c}. 
\end{proof}

%%%%%%%%%%%%%%%%%%%%%%%%%%%%%%%%%%%%%%%%%%%%%%%%%%%%%%%%%%%%%%%%%%%%%%%%%%
\subsection{Joseph-Mathieu-Polo relations for Demazure modules} \label{ss:Mrel}

We now explain the reason why we call $D_\mathbb F^\sigma (\ell, \lambda)$ Demazure modules.  In order to do that, we need the concepts of weight vectors, weight spaces, weight modules and integrable modules for $U_\bb F(\hlie g')$ which are similar to those for $U_\bb F (\lie g_0)$ (cf. Subsection~\ref{sec:rev.g}) by replacing $P_0$ by $\hat P'$. Also, using an analogue of \eqref{eq:P.in.U_Fh}, we obtain an inclusion $\hat P'\hookrightarrow U_\mathbb F(\hlie h')^*$. Let $V$ be a $\bb Z$-graded $U_\bb F(\hlie g')$-module whose weights are in $\hat P'$. As before, let $V[r]$ denote the $r$-th graded piece of $V$. For $\mu\in\hat P$, say $\mu = \mu' + n\delta$ with $\mu' \in \hat P'$, $n \in \bb Z$, denote
\[
V_\mu = \{ v \in V[n] \mid hv=\mu'(h)v \textup{ for all } h\in U_\bb F(\hlie h')\}
\qquad \textup{and} \qquad
\wt(V) = \{\mu\in\hat P \mid V_\mu\ne 0\}.
\]

The next result is a partial twisted affine analogue of Theorem \ref{t:rh}.

\begin{theorem} \label{t:rha}
Let $V$ be a graded $U_\mathbb F(\hlie g')$-module.
\begin{enumerate}[(a), leftmargin=*]
\item \label{t:rha.a}
If $V$ is integrable, then $V$ is a weight-module and $\wt(V) \subseteq\hat P$. Moreover, $\dim V_\mu = \dim V_{w\mu}$ for all $w \in \widehat{\cal W}$, $\mu \in \hat P$.

\item If $V$ is a highest-weight module of highest weight $\lambda \in \hat P^+$, then $\dim(V_{\lambda})=1$ and $V_{\mu}\ne 0$ only if $\mu\le \lambda$. Moreover, $V$  has a unique maximal proper submodule and, hence, also a unique irreducible quotient. In particular, $V$ is indecomposable.

\item{\label{t:rh.ca}} Let $\Lambda\in\hat P^+$ and $n=\Lambda(d)$. Then, the $U_\mathbb F(\hlie g')$-module $\widehat{M}_\bb F (\Lambda)$ generated by a vector $v$ of degree $n$ satisfying the defining relations
\begin{equation*}
U_\mathbb F (\hlie n^+)^0v=0, \quad
hv=\Lambda(h)v \quad \textup{and} \quad 
(x_{\alpha, 0}^-)^{(k)}v=0, \quad
\text{for all $h\in U_\bb F(\hlie h')$, $\alpha \in R_0^+$, $k>\Lambda(h_{\alpha, 0})$},
\end{equation*}
is nonzero and integrable.  Furthermore, every integrable highest-weight module of highest weight $\Lambda$ is a quotient of $\widehat{M}_\bb F (\Lambda)$.
\hfill \qed
\end{enumerate}
\end{theorem}

Given $\Lambda \in \hat P^+$, traditionally, a Demazure module is defined to be the $U_\bb F(\hlie b'^+)$-submodule $V_\mathbb F^w(\Lambda) \subseteq \widehat{M}_{\bb F} (\Lambda)$ generated by the weight space $\widehat{M}_\bb F (\Lambda)_{w\Lambda}$ for some $w \in \widehat{\cal W}$ (cf. \cite{mathieu89, foli:weyldem, naoi:weyldem, FK13}).  Our focus is on Demazure modules that are stable under the action of $U_\bb F(\lie g_0)$.  Since $V_\bb F^w(\Lambda)$ is defined as a $U_\bb F (\hlie b'^+)$-module, it is stable under the action of $U_\bb F(\lie g_0)$ if, and only if,
\begin{equation}\label{e:n-inv}
U_\mathbb F(\lie n_0^-)^0 \widehat{M}_\bb F(\Lambda)_{w\Lambda}=0 .
\end{equation}
In particular, since $V_\bb F^w(\Lambda)$ is an integrable $U_\bb F(\lie{g}_0)$-module, it follows that $(w\Lambda)(h_{\alpha,0}) \le 0$ for all $\alpha \in R_0^+$.  Conversely, using the exchange condition for Coxeter groups (see \cite[Section 5.8]{humphreys90}), for all $i \in \hat I$, we have
\[
(x_{\alpha_i}^\epsilon)^{(k)} \widehat{M}_\bb F(\Lambda)_{w\Lambda} = 0
\qquad \textup{for all} \quad k>0,
\]
where $\epsilon=+$ if $w\Lambda(h_{\alpha_i}) \ge 0$ and $\epsilon=-$ if $w\Lambda(h_{\alpha_i})\le 0$. Thus, if $w\Lambda (h_{\alpha, 0}) \le 0$ for all $\alpha \in R_0^+$, then $V_\bb F^w(\Lambda)$ is $U_\bb F(\lie g_0)$-stable.

Henceforth, assume that $(w\Lambda)(h_{\alpha,0}) \le 0$ for all $\alpha \in R^+_0$, and observe that this implies that $w\Lambda$ must have the form
\begin{equation}\label{e:thelevel}
w\Lambda = \ell \Lambda_0 - \lambda + n \delta
\qquad \textup{for some $\lambda \in P_0^+$, $n \in \bb Z$, and $\ell = \Lambda(c) \ge 0$}.
\end{equation}
Conversely, given $\ell \in\bb Z_{\ge 0}$, $\lambda\in P_0^+$ and $n \in \bb Z$, since $\widehat{\cal W}$ acts simply transitively on the set of alcoves of $\hlie h^\ast$ (see \cite[Theorem 4.5.(c)]{humphreys90}), there exists a unique $\Lambda \in \hat{P}^+$ such that $\ell \Lambda_0 - \lambda + n \delta \in \widehat{\cal W} \Lambda$. Thus, if $w \in \widehat{\cal W}$ and $\Lambda \in \hat{P}^+$ are such that
\begin{equation}\label{e:notcon}
w\Lambda = \ell\Lambda_0 - \lambda + n \delta,
\end{equation}
then $V_\bb F^w(\Lambda)$ is $U_\bb F(\lie g_0)$-stable.  Henceforth, we fix $w$, $\Lambda$, $\lambda$ and $n$ as in \eqref{e:notcon}.
Notice that, if $\gamma = \pm \alpha + s \delta \in \hat R^+$, $\alpha \in R_0^+$, $s \ge 0$, then
\[
w\Lambda(h_\gamma) = 
\begin{cases}
\mp \delta_{s, 1} \lambda(h_{\frac{\alpha}{2}, 0}) + s \ell \hat r^\vee_\alpha, & \textup{if $\lie g \cong A_{2n}$ and $\alpha \in 2 R_{\sh}$}, \\
\mp \lambda(h_{\alpha,0}) + s \ell \hat r^\vee_\alpha, & \textup{otherwise}.
\end{cases}
\]

The following lemma is a rewriting of \cite[Lemme~26]{mathieu89} using the above fixed notation. (Compare it with \cite[Theorem~1]{foli:weyldem}, \cite[Proposition~4.8]{FK13}  and \cite[Lemma~3.5.3]{BMM15}.)

\begin{lemma}\label{l:jpmrel} 
The $U_\bb F(\lie g[t]^\sigma)$-module $V_\bb F^w(\Lambda)$ is isomorphic to a $U_\bb F(\lie g[t]^\sigma)$-module generated by a vector $v$ of degree $n$ satisfying the following defining relations:
\begin{alignat}{2}\label{e:MFLrel+}
&U_\bb F(\lie n^-[t]^\sigma)^0v
= U_\bb F(\lie h [t]_+^\sigma)^0v
= &&0, \qquad \qquad
hv = -\lambda(h)v \quad \textup{for all $h \in U_\bb F(\lie h)$}, \\
&(x_{\alpha,-s}^+ \otimes t^{s})^{(k)}v=0
\quad \textup{for all} \quad && \alpha\in R_0^+,\ s \ge 0, \notag \\
& &&k > 
\begin{cases}
\max \{ 0, \delta_{s, 1} \lambda(h_{\frac{\alpha}{2}, 0}) - s \ell \hat r^\vee_\alpha \}, & \textup{if $\lie g \cong A_{2n}$ and $\alpha \in 2 R_{\sh}$}, \\
\max \{ 0, \lambda(h_{\alpha,0}) - s \ell \hat r^\vee_\alpha \}, & \textup{otherwise}. \notag
\end{cases}
\end{alignat} 
\end{lemma}

The following is the main result of this subsection.

\begin{proposition}\label{p:Demrels}
The graded $U_\bb F(\lie g[t]^\sigma)$-modules $V_\bb F^w(\Lambda)$ and $D_\bb F^\sigma(\ell, \lambda, n)$ are isomorphic.
\end{proposition}

\begin{proof}
It suffices to prove the statement for $n=0$ and, thus, for simplicity, we assume that this is the case.  First, we will show that $V_\bb F^w(\Lambda)$ is a quotient of $D_\bb F^\sigma (\ell,\lambda)$.  Let $\mu = \ell \Lambda_0 + \lambda$.  Since $\widehat{M}_\bb F(\Lambda)$ is integrable and $V_\bb F^w (\Lambda)$ is a $U_\bb F (\lie g_0)$-submodule of $\widehat{M}_\bb F (\Lambda)$, it follows from Theorem~\ref{t:rha}\ref{t:rha.a} that
\[
\dim V_\bb F^w (\Lambda)_\mu
= \dim V_\bb F^w (\Lambda)_{w_0\mu}
= \dim \widehat{M}_\bb F (\Lambda)_{w\Lambda}
= \dim \widehat{M}_\bb F (\Lambda)_{\Lambda} = 1.
\]
Hence there exists a nonzero vector $v \in V_\bb F^w (\Lambda)_\mu$, which is extremal and generates $V_\bb F^w(\Lambda)$.  Since $v$ is an extremal weight vector of weight $\mu$ in $\widehat{M}_\bb F(\Lambda)$, we have, for all positive real roots $\gamma\in\hat R^+$,
\begin{equation} \label{eq:xwv?}
(x^+_\gamma)^{(k)}v = 0 \quad\text{for all}\quad k>\max\{0,-\mu(h_\gamma)\}.
\end{equation}
In particular, by taking $\gamma = \alpha + s \delta$ with $\alpha \in R_{-s}^+$ and $s \ge 0$, we obtain that
\begin{align*}
-\mu(h_\gamma) = -\delta_{s,1} \lambda(h_{\frac{\alpha}{2}, 0}) - \ell \hat r_\alpha^\vee s \le 0, \qquad
&\textup{in case $\lie g \cong A_{2n}$ and $\alpha \in 2R_{\sh}$}, \\
-\mu(h_\gamma) = -\lambda(h_{\alpha, 0}) - \ell \hat r_\alpha^\vee s \le 0, \qquad
&\textup{in every other case}.
\end{align*}
This implies that $(x_{\alpha, -s}^+ \otimes t^s)^{(k)} v = 0$ for all $k>0$, $s \ge 0$, $\alpha \in R^+_{-s}$.  Thus $U_\bb F (\lie n^+[t]^\sigma)^0 v = 0$.  Similarly, by taking $\gamma = -\alpha + s \delta$ with $\alpha \in R_{-s}^+$ and $s \ge 0$, we obtain that
\begin{align*}
-\mu(h_\gamma) = \delta_{s,1} \lambda(h_{\frac{\alpha}{2}, 0}) - \ell \hat r_\alpha^\vee s \le 0, \qquad
&\textup{in case $\lie g \cong A_{2n}$ and $\alpha \in 2R_{\sh}$}, \\
-\mu(h_\gamma) = \lambda(h_{\alpha, 0}) - \ell \hat r_\alpha^\vee s \le 0, \qquad
&\textup{in every other case}.
\end{align*}
This implies that $(x_{\alpha, -s}^- \otimes t^s)^{(k)} v = 0$ for all $k > \max\{ 0, \delta_{s,1} \lambda(h_{\frac{\alpha}{2}, 0}) - \ell \hat r_\alpha^\vee s\}$, in case $\lie g \cong A_{2n}$, $\alpha \in 2 R_{\sh}$, and implies that $(x_{\alpha, -s}^- \otimes t^s)^{(k)} v = 0$ for all $k > \max\{ 0, \lambda(h_{\alpha, 0}) - \ell \hat r_\alpha^\vee s\}$ in every other case.  These vanishing conditions for $(x_{\alpha, -s}^- \otimes t^s)^{(k)} v$ and $(x_{\alpha, -s}^+ \otimes t^s)^{(k)} v$, together with Lemma~\ref{basicreltw} imply that $U_\bb F(\lie h[t]^\sigma_+)^0v=0$.  Thus, $v$ is a generator of $V_\bb F^w (\Lambda)$ satisfying the defining relations \eqref{eq:defn.rel.dem}.  This implies that $V_\bb F^w(\Lambda)$ is a quotient of $D_\bb F^\sigma (\ell,\lambda)$.

It now suffices to show that $\dim D_\bb F^\sigma (\ell,\lambda) \le \dim V_\bb F^w (\Lambda)$.  In fact, we will show that $D'$, the pull-back of $D_\bb F^\sigma (\ell, \lambda)$ by $\psi$, is a quotient of $V_\bb F^w(\Lambda)$.  Let this time $v$ be a nonzero generator in $D_\bb F^\sigma (\ell,\lambda)_\lambda$ and $v'$ denote $v$ when regarded as an element of $D'$.  Since $U_\bb F(\lie n^+[t]^\sigma)^0 v = 0$ and $\psi \left( U_\bb F(\lie n^-[t]^\sigma)^0 \right) = U_\bb F(\lie n^+[t]^\sigma)^0$, it follows that $U_\bb F(\lie n^-[t]^\sigma)^0 v' = 0$.  Since $\psi$ restricts to an automorphism of $U_\bb F(\lie h[t]^\sigma_+)$ and $U_\bb F (\lie h[t]^\sigma_+)^0 v = 0$, then $U_\bb F (\lie h[t]^\sigma_+)^0 v' = 0$.  Moreover, since $h v = \lambda(h)v$ for all $h\in U_\bb F(\lie h)$, \eqref{e:inv+-w} implies that $hv' = -\lambda(h)v'$ for all $h \in U_\bb F(\lie h)$.  Finally, the defining relations \eqref{eq:defn.rel.dem} for $v$ and \eqref{e:inv+-def} imply that
\[
(x_{\alpha,s}^+ \otimes t^s)^{(k)} v'
= (x_{\alpha,s}^- \otimes t^s)^{(k)} v 
= 0 
\qquad \textup{for all \ $\alpha \in R_0^+$, $s \ge 0$, $k > \max\{0, \delta_{s,1}\lambda(h_{\frac{\alpha}{2},0}) - s \ell \hat r^\vee_\alpha\}$},
\]
in case $\lie g \cong A_{2n}$ and $\alpha \in 2R_{\sh}$, and imply that 
\[
(x_{\alpha,s}^+ \otimes t^s)^{(k)} v'
= (x_{\alpha,s}^- \otimes t^s)^{(k)} v 
= 0 
\qquad \textup{for all \ $\alpha \in R_0^+$, $s \ge 0$, $k > \max\{0, \lambda(h_{\alpha,0}) - s \ell \hat r^\vee_\alpha\}$},
\]
in every other case.  Thus $v'$ satisfies the defining relations of the generator of $V_\bb F^w(\Lambda)$ given in Lemma \ref{l:jpmrel}.  This shows that $D'$ is a quotient of $V_\bb F^w (\Lambda)$.  Therefore, $\dim D_\bb F^\sigma (\ell, \lambda) = \dim D' \le \dim V_\bb F^w (\Lambda)$.
\end{proof}

%%%%%%%%%%%%%%%%%%%%%%
\section{Main results}
 
\subsection{Connection between twisted Weyl modules and twisted Demazure modules}
	
In this section we will show that almost all Weyl modules are isomorphic to certain Demazure modules. 

The following lemma will be used in Theorem~\ref{t:wd}.

\begin{lemma}\cite[Proposition 2.5.4]{JM2}.  \label{l:nilsl2t0}
	Let $m\in\mathbb Z_{\ge0}$ and let $v$ be a nonzero vector of weight $m\omega_1$ of the $U_\mathbb F(\lie{sl}_2[t])$-module $W_\mathbb F^c(m)$. Then,
	$(x_{\alpha_1}^-\otimes t^r)^{(k)}v=0$ 
	for all $r\in\mathbb Z_{\ge 0}$, $k > max\{0, m - r\}$.\hfill\qedsymbol
\end{lemma}

\begin{theorem}\label{t:wd}
Suppose $\lie g$ is of type $A_{2l-1}^{(2)}, D_{l+1}^{(2)}, E_{6}^{(2)}$ or $D_{4}^{(3)}$. Then, we have and isomorphism of $U_\mathbb F(\lie g[t]^\sigma)$-modules $$W_\mathbb F^{c,\sigma}(\lambda) \cong D_\mathbb F^\sigma(1,\lambda)$$
\end{theorem}

\proof  It follows from the definitions in Section \ref{s:wdm} that the Demazure module $D_\mathbb F^\sigma(1,\lambda)$ is a quotient of the the Weyl module $W_\mathbb F^{c,\sigma}(\lambda)$. By comparing the defining relations in Lemma \ref{l:jpmrel} and Definition \ref{d:weyl}, to prove that these modules are isomorphic it is suffices to show that the generator of the Weyl module is subject to the following relations:
$$(x_{\alpha,-s}^+ \otimes t^{s})^{(k)}v=0
\quad \textup{for all} \quad \alpha\in R_0^+,\ s \ge 0, \ k > \max\{0, \lambda(h_{\alpha,0})- s  \hat r^\vee_\alpha\}$$
Notice that these relations are equivalent to:
\begin{equation}
(x_{\alpha,-j}^+ \otimes t^{ms+j})^{(k)}v=0
\end{equation}
for all $\alpha\in R_0^+,\ 0\le j\le m-1, \ s \ge 0,\  k > \begin{cases} \max\{0, \lambda(h_{\alpha,0})- s \},  \mbox{ if } \alpha \mbox{ is long } \\ \max\{0, \lambda(h_{\alpha,0})- (ms+j)\},  \mbox{ if } \alpha \mbox{ is short }. \end{cases}$

Let $\alpha\in R_0^+$ be a long root and $V=U_\mathbb F(\lie{sl}_{2,\alpha}[t^m])\cdot v\subseteq W_\mathbb F^{c,\sigma}(\lambda)$. 
Notice that $v$ is a cyclic generator for $V$ and satisfies the defining relations of of the nontwisted graded Weyl $U_\mathbb F(\lie{sl}_{2}[t])$-module $W_\mathbb F^c(\lambda(h_{\alpha,0})\omega_1)$. Thus, we conclude that $V$ is a quotient $W_\mathbb F^c(\lambda(h_{\alpha,0})\omega_1)$. In particular, by Lemma~\ref{l:nilsl2t0}, $v$ satisfies the relations 
\begin{equation}\label{sl2wdrel}
(x_{\alpha_1}^{-}\otimes t^s)^{(k)}.v=0\ \text{ for all } s\in \mathbb Z_{\geq 0} \text{ and } k>\max\{ 0 ,\lambda(h_{\alpha_1})-s\}.
\end{equation}
Now, by Lemma~\ref{isos} we obtain 
$$(x^{-}_{\alpha,0} \otimes t^{ms})^{(k)}\cdot v = 0 \text{ for all } s\in \mathbb Z_{\geq 0} \text{ and } k>\max\{ 0 ,\lambda(h_{\alpha,0})-s\}.$$
Finally, suppose $\alpha$ is a short root and consider the $U_\mathbb F(\lie{sl}_{2,\alpha}[t])$-submodule $V=U_\mathbb F(\lie{sl}_{2,\alpha}[t])\cdot v \subseteq W_\mathbb F^{c,\sigma}(\lambda)$. In a similar fashion as above we can use Lemma~\ref{isos} to conclude that $V$ is a quotient of $W_\mathbb F^c(\lambda(h_{\alpha,0})\omega_1)$ 
and, therefore, $v$ satisfies the relations in \eqref{sl2wdrel}. So, by using the isomorphism in Lemma~\ref{isos} we obtain 
\[(x_{\alpha,-j}^+ \otimes t^{ms+j})^{(k)}v=0 \text{ for all } \alpha\in R_0^+,\ 0\le j\le m-1, \ s \ge 0,\  k > \max\{0, \lambda(h_{\alpha,0})- (ms+j)\}.\qedhere\]
\endproof

\subsection{Connection between twisted and untwisted Weyl modules}

In this section we will show that the twisted Weyl modules can be realized as modules for the hyper current algebra constructed from certain untwisted Weyl modules. The version of this result for hyper loop algebra was proved in \cite{BM14} and the methods are completely different.

Suppose that $\rm{char}(\mathbb F)\ne m$. Denote by $\textrm{res}^{U_\bb F (\lie g[t])}_{U_\bb F (\lie g[t]^\sigma)} W_\bb F^c (\lambda)$ the module obtained by regarding $W_\bb F^c (\lambda)$ as a $U_\mathbb F(\lie g[t]^\sigma)$-module via restriction of the action of $U_\mathbb F(\lie g[t])$ to $U_\mathbb F(\lie g[t]^\sigma)$. 

\begin{theorem} \label{thm.2}
There is an isomorphism $\textrm{res}^{U_\bb F (\lie g[t])}_{U_\bb F (\lie g[t]^\sigma)} W_\bb F^c (\lambda) \cong W_\bb F^{c, \sigma} (\lambda)$ of $U_\bb F (\lie g[t]^\sigma)$-modules.
\end{theorem}

\begin{proof}
Let $v\in W_\bb F^c (\lambda)$ be a cyclic generator and $W=U_\bb F (\lie g[t]^\sigma)v\subseteq U_\bb F (\lie g[t])v=W_\bb F^c (\lambda)$. Under these assumptions we have $W=\textrm{res}^{U_\bb F (\lie g[t])}_{U_\bb F (\lie g[t]^\sigma)} W_\bb F^c (\lambda)$. We will show that $W$ is a quotient of $W_\bb F^{c, \sigma} (\lambda)$. 

Firstly, from the defining relations of $W_\bb F^c (\lambda)$ and the construction of $\cal C^\sigma(O)$, we easily conclude that 
$$U_\bb F (\lie n^+[t]^\sigma)^0v=U_\bb F (\lie h[t]_+^\sigma)^0v=0.$$
It remains to show that 
$$h v= \lambda(h)v \quad \textup{ and } \quad (x_{\alpha,0}^- \otimes 1)^{(k)}v=0,$$
for all $h \in U_\bb F (\lie h_0)$, $\alpha \in R_0^+$, and $k > \lambda( h_{\alpha,0} )$. From Section \ref{s:chevalley}, we have
$$\lambda(h_{\alpha,0})=\begin{cases} 
\lambda(h_{\alpha}), & \text{ if } \lie g \text{ is not of type } A_{2n} \text{ and } \alpha=\sigma(\alpha),\\
\lambda(h_{\alpha}+h_{\sigma(\alpha)}), & \text{ if } \lie g \text{ is not of type } A_{2n} \text{ and } \alpha\ne\sigma(\alpha),\\
\lambda(2(h_{\alpha}+h_{\sigma(\alpha)})), & \text{ if } \lie g \text{ is of type } A_{2n} \text{ and } \alpha\mid_{\lie h_0}\in R_{\sf sh},
\end{cases}$$
and
$$h_{\alpha,0}v=\begin{cases} 
h_{\alpha}v, & \text{ if } \lie g \text{ is not of type } A_{2n} \text{ and } \alpha=\sigma(\alpha),\\
(h_{\alpha}+h_{\sigma(\alpha)})v, & \text{ if } \lie g \text{ is not of type } A_{2n} \text{ and } \alpha\ne\sigma(\alpha),\\
2(h_{\alpha}+h_{\sigma(\alpha)})v, & \text{ if } \lie g \text{ is of type } A_{2n} \text{ and } \alpha\mid_{\lie h_0}\in R_{\sf sh},
\end{cases}$$
which leads to $h v= \lambda(h)v$ for all $h \in U_\bb F (\lie h_0)$ due to the defining relations of $W_\bb F^c (\lambda)$. Also from Section \ref{s:chevalley} we have 
$$(x_{\alpha,0}^- \otimes 1)^{(k)}v=\begin{cases} 
(x_{\alpha}^- \otimes 1)^{(k)}v, & \text{ if } \lie g \text{ is not of type } A_{2n} \text{ and } \alpha=\sigma(\alpha),\\
((x_{\alpha}^-+x_{\sigma(\alpha)}^-) \otimes 1)^{(k)}v, & \text{ if } \lie g \text{ is not of type } A_{2n} \text{ and } \alpha\ne\sigma(\alpha),\\
(\sqrt{2}(x_{\alpha}^-+x_{\sigma(\alpha)}^-) \otimes 1)^{(k)}v, & \text{ if } \lie g \text{ is of type } A_{2n} \text{ and } \alpha\mid_{\lie h_0}\in R_{\sf sh}\end{cases}$$
and  supposing $k>\lambda(h_{\alpha,0})$, from the defining relations of $W_\bb F^c (\lambda)$, we conclude that $(x_{\alpha,0}^- \otimes 1)^{(k)}v=0$ 
for all these cases, since the first case is direct, the second case follows by taking the binomial expansion $$(x_{\alpha}^-+x_{\sigma(\alpha)}^-)^{(k)}=\sum_{\substack{k_1,k_2\in\mathbb Z_{\ge0}\\ k_1+k_2=k}} (x_{\alpha}^-)^{(k_1)}(x_{\sigma(\alpha)}^-)^{(k_2)}$$ where we observe that either $k_1>\lambda(h_\alpha)$ or $k_2>\lambda(h_{\sigma(\alpha)})$ and $x_{\alpha}^-$ and $x_{\sigma(\alpha)}^-$ commutes, and the third case follows from the expansion  
$$(x_{\alpha}^-+x_{\sigma(\alpha)}^-)^{(k)}=\sum_{\substack{0\le n \le k\\ n\equiv_2 k }} (-\frac{1}{2}x_{\alpha+\sigma(\alpha)}^-)^{(\frac{k-n}{2})} \sum_{\substack{k_1,k_2\in\mathbb Z_{\ge0}\\ k_1+k_2=n}} (x_{\alpha}^-)^{(k_1)}(x_{\sigma(\alpha)}^-)^{(k_2)}$$
and the fact that $\lambda - 2\alpha - 2\sigma(\alpha) \notin \wt(W_\bb F^c (\lambda))$ for any $\alpha \in R_0^+$.
\end{proof}

%%%%%%%%%%%%%%%%%%%%%%%%%%%%%%%

\end{document}